\documentclass[12pt]{article}
\usepackage{fullpage}
\usepackage{amsmath,amssymb,amsfonts,amsthm}
\usepackage[all,arc,knot,poly]{xy}
\usepackage{enumerate}
\usepackage{mathrsfs}
\usepackage{mathtools}
\usepackage{amsxtra}
\usepackage{needspace}
\usepackage{graphicx}
\usepackage{setspace}

\DeclareMathOperator{\Mod}{\operatorname{mod}}
\DeclareMathOperator{\Gr}{\operatorname{Gr}}
\DeclareMathOperator{\Stab}{\operatorname{Stab}}
\DeclareMathOperator{\Sph}{\operatorname{Sph}}
\DeclareMathOperator{\sgn}{\operatorname{sgn}}
\DeclareMathOperator{\Tr}{\operatorname{Tr}}
\DeclareMathOperator{\Aut}{\operatorname{Aut}}
\DeclareMathOperator{\Ext}{\operatorname{Ext}}

\DeclareMathOperator{\Hom}{\operatorname{Hom}}
\DeclareMathOperator{\Rep}{\operatorname{Rep}}

\theoremstyle{definition}
\newtheorem{theorem}{Theorem}[section]
\newtheorem{definition}[theorem]{Definition}
\newtheorem{lemma}[theorem]{Lemma}
\newtheorem{proposition}[theorem]{Proposition}

\newtheorem{assumption}[theorem]{Assumption}

\title{Categorified canonical bases and framed BPS states}

\author{
  Dylan G.L. Allegretti
}

\date{}

\begin{document}

\maketitle

\begin{abstract}
We consider a cluster variety associated to a triangulated surface without punctures. The algebra of regular functions on this cluster variety possesses a canonical vector space basis parametrized by certain measured laminations on the surface. To each lamination, we associate a graded vector space, and we prove that the graded dimension of this vector space gives the expansion in cluster coordinates of the corresponding basis element. We discuss the relation to framed BPS states in $\mathcal{N}=2$ field theories of class~$\mathcal{S}$.
\end{abstract}

\tableofcontents

\section{Introduction}

\subsection{Canonical bases and categorification}

Distinguished bases for the coordinate rings of various algebraic spaces have been the subject of intense research in representation theory since the pioneering work of Lusztig~\cite{Lusztig}. These canonical bases do not depend on any arbitrary choices, they can be naturally $q$-deformed, and they have remarkable positivity properties.

In this paper, we study a particular kind of canonical basis discovered by Fock and Goncharov~\cite{FG1}. The starting point for this construction is a compact oriented surface~$\mathbb{S}$ with a finite set $\mathbb{M}$ of marked points such that every boundary component of~$\mathbb{S}$ contains a marked point. Such a pair $(\mathbb{S},\mathbb{M})$ is called a \emph{marked bordered surface}. In general, one can consider marked points in the interior of the surface, also known as \emph{punctures}, but in this paper we will assume that all marked points lie on the boundary. In Section~\ref{sec:GeneralizationToArbitrarySurfaces}, we comment on the possible generalization to an arbitrary marked bordered surface.

Let $\Sigma=(\mathbb{S},\mathbb{M})$ be a marked bordered surface without punctures. In~\cite{FG1}, Fock and Goncharov defined a moduli space $\mathcal{X}_{\Sigma,PGL_2}$ parametrizing $PGL_2$-local systems on the surface with additional data associated to the marked points. This moduli space has an atlas of coordinate charts corresponding to triangulations of the surface. More precisely, an \emph{ideal triangulation} of $\Sigma$ is defined as a triangulation of the surface all of whose edges begin and end at marked points. For any ideal triangulation $T$ of~$\Sigma$, Fock and Goncharov defined a collection of rational coordinates 
\[
X_i:\mathcal{X}_{\Sigma,PGL_2}\dashrightarrow\mathbb{G}_m,
\]
indexed by the internal edges $i$ of~$T$. These coordinates are called \emph{cluster coordinates}. The \emph{cluster Poisson variety}~$\mathcal{X}$ is the open subset of $\mathcal{X}_{\Sigma,PGL_2}$ consisting of points at which these coordinates are regular for some ideal triangulation.

Fock and Goncharov also considered a certain kind of measured lamination on the surface~$\mathbb{S}$. Such a lamination is given by a collection of finitely many nonintersecting curves on~$\mathbb{S}$ with integer weights. These curves may be closed, or they may connect points on the boundary of $\mathbb{S}$ away from the marked points, and they are subject to certain axioms and equivalence relations. Figure~\ref{fig:lamination} shows an example in the case where $\Sigma$ is a one-holed torus with four marked points on its boundary.
\begin{figure}[ht]
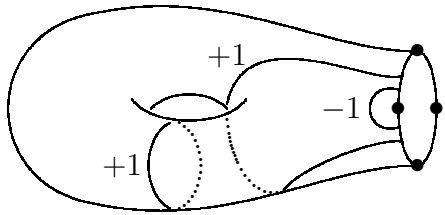
 \begin{center}
\[
\xy 0;/r.6pc/: 
(6,0)*\ellipse(1,3){-};
(-5,5)*{}="U1";
(-5,-5)*{}="L1";
(12,3)*{}="U2";
(12,-3)*{}="L2";
(-2,0)*{}="1";
(2,0)*{}="2";
(-3,0.5)*{}="3";
(3,0.5)*{}="4";
"U1";"L1" **\crv{(-11,4) & (-11,-4)};
"U1";"U2" **\crv{(2,6.5) & (7,3)};
"L1";"L2" **\crv{(2,-6.5) & (7,-3)};
"1";"2" **\crv{(-1,1) & (1,1)};
"3";"4" **\crv{(-2,-1) & (2,-1)};
(12,3)*{\bullet}, 
(12,-3)*{\bullet}, 
(11,0)*{\bullet}, 
(13,0)*{\bullet}, 
(11,1);(11,-1) **\crv{(9,1.5) & (9,-1.5)};
(11.25,1.75);(2,0.25) **\crv{(10,1) & (3,5)};
(11.25,-1.75);(5,-4.25) **\crv{(10,-1.5) & (6,-3)};
(5,-4.3);(2,-0.25) **\crv{~*=<2pt>{.} (3,-5) & (2,-2)};
(-1,-0.75);(-1,-5.25) **\crv{(-2.5,-1.5) & (-2.5,-4.5)};
(-0.5,-0.75);(-0.5,-5.25) **\crv{~*=<2pt>{.} (1,-1.5) & (1,-4.5)};
(-3.5,-3)*{+1}, 
(2,2.75)*{+1}, 
(8,0)*{-1}, 
\endxy
\]
\caption{A lamination on a one-holed torus with four marked points.\label{fig:lamination}}
\end{center} \end{figure}

For each lamination $\ell$, Fock and Goncharov defined a rational function 
\[
\mathbb{I}(\ell)\in\mathbb{Q}(\mathcal{X}_{\Sigma,PGL_2})
\]
on $\mathcal{X}_{\Sigma,PGL_2}$. For example, if $\ell$ is a lamination consisting of a single loop of weight~$k>0$, then $\mathbb{I}(\ell)$ is defined by taking the trace of the $k$th power of the monodromy around the loop. Since the moduli space $\mathcal{X}_{\Sigma,PGL_2}$ and cluster variety $\mathcal{X}$ are birational, $\mathbb{I}(\ell)$ can also be viewed as a rational function on~$\mathcal{X}$ which in fact turns out to be regular. In Section~\ref{sec:TheCanonicalBasisConstruction}, we prove the following statement, extending the result of Fock and Goncharov for punctured surfaces without boundary (\cite{FG1}, Theorem~12.3).

\begin{theorem}
\label{thm:introcanonicalbasis}
If $|\mathbb{M}|>1$ then the functions $\mathbb{I}(\ell)$ form a canonical vector space basis for the algebra $\mathcal{O}(\mathcal{X})$ of regular functions on the cluster Poisson variety.
\end{theorem}
 
The proof of this result is based on ideas of Musiker, Schiffler, and Williams~\cite{MSW2} who defined canonical bases for cluster algebras arising from triangulated surfaces. A more general method of constructing canonical bases using tools from mirror symmetry was recently developed by Gross, Hacking, Keel, and Kontsevich~\cite{GHKK}, and their results have been used by Goncharov and Shen~\cite{GS} to describe canonical bases for the coordinate rings of cluster varieties associated to more general surfaces. In this paper, rather than work with the abstractly defined bases of~\cite{GHKK}, we focus on bases defined concretely in terms of trace functions on moduli spaces of local systems.

The functions $\mathbb{I}(\ell)$ which give the canonical basis can be written in a particularly nice way using the coordinates described above. Indeed, suppose $T$ is an ideal triangulation of~$\Sigma$ and $\Gamma_T=\bigoplus_{j\in J}\mathbb{Z}\mathbf{e}_j$ is the lattice with basis vectors $\mathbf{e}_j$ indexed by the set $J=J_T$ of internal edges of~$T$. Then each function $\mathbb{I}(\ell)$ can be expanded as 
\begin{align}
\label{eqn:expansion}
\mathbb{I}(\ell)=\sum_{\mathbf{d}\in\Gamma_T}c_{\mathbf{d}} X_{\mathbf{d}}
\end{align}
where the coefficients $c_{\mathbf{d}}$ are nonnegative integers and we write $X_{\mathbf{d}}=\prod_{j\in J}X_j^{d_j}$ for each vector $\mathbf{d}=\sum d_j\mathbf{e}_j$. It is natural to suspect that the canonical basis described above should possess some sort of categorification. More precisely, the expansion~\eqref{eqn:expansion} should arise as the graded dimension of some naturally defined graded vector space. The goal of this paper is to give a representation theoretic construction of such a vector space and show that it is closely related to the mathematical description of line defects in certain supersymmetric quantum field theories.

\subsection{The main construction}

Let us describe the main idea of our construction in more detail. Fix an ideal triangulation~$T$ of the marked bordered surface $\Sigma=(\mathbb{S},\mathbb{M})$. From this triangulation, one can construct an associated quiver~$Q(T)$. Roughly speaking, this quiver is obtained by drawing a vertex dual to each internal edge of the triangulation and connecting these vertices by arrows within each of the triangles. Figure~\ref{fig:quivertriangulation} shows an example in the case where $T$ is an ideal triangulation of a disk with eight marked points on its boundary.
\begin{figure}[ht]
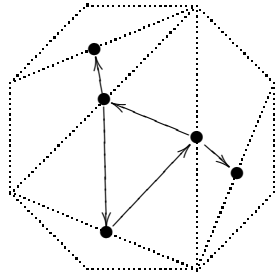
 \begin{center}
\[
\xy /l1.5pc/:
{\xypolygon8"A"{~:{(3,0):}~>{.}}},
{"A1"\PATH~={**@{.}}'"A3"'},
{"A1"\PATH~={**@{.}}'"A6"'},
{"A3"\PATH~={**@{.}}'"A6"'},
{"A3"\PATH~={**@{.}}'"A5"'},
{"A8"\PATH~={**@{.}}'"A6"'},
(-0.15,0)*{\bullet}="a", 
(1.8,-0.8)*{\bullet}="b",
(1.75,2)*{\bullet}="c",
(2,-1.85)*{\bullet}="d",
(-1,0.75)*{\bullet}="e", 
\xygraph{
"a":"b",
"b":"c",
"c":"a",
"b":"d",
"a":"e",
}
\endxy
\]
\caption{The quiver associated to an ideal triangulation.\label{fig:quivertriangulation}}
\end{center} \end{figure}
The quiver~$Q(T)$ comes with a canonical potential $W(T)$ defined by Labardini-Fragoso~\cite{L1}.

The \emph{Jacobian algebra} $J(Q,W)$ of a quiver with potential $(Q,W)$ is defined as a quotient of the completed path algebra of the quiver by certain relations coming from the potential~\cite{DWZ1}. For the quivers with potential that we consider, the Jacobian algebra is known to be finite-dimensional~\cite{L1}, and therefore we can apply a construction of Amiot~\cite{Amiot} to get a 2-Calabi-Yau triangulated category known as the \emph{cluster category} $\mathcal{C}(Q,W)$. It follows from results of Keller and Yang~\cite{KellerYang} that if $T$ and $T'$ are two ideal triangulations of $\Sigma$, then the categories $\mathcal{C}(Q(T),W(T))$ and $\mathcal{C}(Q(T'),W(T'))$ are equivalent, and therefore there is a cluster category~$\mathcal{C}$ canonically associated to the marked bordered surface $\Sigma$.

A result of Br\"ustle and Zhang~\cite{BrustleZhang} gives an explicit parametrization of the indecomposable objects of this cluster category~$\mathcal{C}$. According to this result, an indecomposable object is either a \emph{string object} corresponding to an arc on $\mathbb{S}$ connecting two marked points, or it is a \emph{band object} corresponding to a loop on the surface with additional decorations. For any ideal triangulation~$T$, there is a functor 
\[
\label{eqn:functor}
\mathcal{C}\rightarrow\Mod J(Q(T),W(T))
\]
from the cluster category to the category of finite-dimensional modules over the Jacobian algebra of~$(Q(T),W(T))$. Applying this functor to a string object or band object gives an indecomposable module over the Jacobian algebra known as a \emph{string module} or \emph{band module}, respectively. Modules associated to arcs on a surface have been studied by many authors because of their relations to cluster variables; see for example~\cite{CCS,ABCP,L2}. Modules associated to loops are less well studied, but they have appeared in~\cite{DupontThomas}.

In Section~4, we associate to each ideal triangulation $T$ and each lamination $\ell$ a module~$M_{T,\ell}$ over the corresponding Jacobian algebra $J(Q(T),W(T))$. It is defined as a direct sum of string and band modules associated to the constituent curves of the lamination. For any vector $\mathbf{d}\in\Gamma_T$ in the lattice defined above, the \emph{quiver Grassmannian} $\Gr_{\mathbf{d}}(M_{T,\ell})$ is a complex projective variety parametrizing submodules of~$M_{T,\ell}$ which, when regarded as representations of the quiver~$Q(T)$, have dimension vector~$\mathbf{d}$. We consider a constructible subset $\Gr_{\mathbf{d}}^\circ(M_{T,\ell})\subseteq\Gr_{\mathbf{d}}(M_{T,\ell})$ generalizing the \emph{transverse quiver Grassmannian} studied in~\cite{Dupont,CerulliEsposito,CerulliDupontEsposito}. Our first main result is the following.

\begin{theorem}
\label{thm:introcategorification1}
Let $T$ be an ideal triangulation of~$\Sigma$ and $\ell$ be a lamination. Then there exists a vector $\mathbf{h}=\mathbf{h}^{T,\ell}\in\Gamma_T$ such that 
\[
\mathbb{I}(\ell)=\sum_{\mathbf{d}\in\Gamma_T, i\in\mathbb{Z}}(-1)^i\dim_{\mathbb{C}}\mathcal{H}^{T,\ell}_{\mathbf{d},i}\cdot X_{\mathbf{d}}
\]
where $\mathcal{H}^{T,\ell}_{\mathbf{d},i}=H^i(\Gr_{\mathbf{d}-\mathbf{h}}^\circ(M_{T,\ell}),\mathbb{C})$ is the singular cohomology of the subset described above in the analytic topology.
\end{theorem}

The importance of quiver Grassmannians in the theory of cluster algebras was first recognized by Caldero and Chapoton~\cite{CalderoChapoton} who showed that these varieties can be used to categorify the generators of certain cluster algebras. Their results were refined and generalized by Derksen, Weyman, and Zelevinsky~\cite{DWZ2} who showed that generating functions for Euler characteristics of quiver Grassmannians coincide with the $F$-polynomials introduced by Fomin and Zelevinsky~\cite{FZIV}. Related ideas have been used by Dupont to define canonical bases for cluster algebras~\cite{Dupont11}. Note that we categorify functions on the cluster Poisson variety rather than the cluster variables studied in these earlier works.

In the statement of Theorem~\ref{thm:introcategorification1}, we prefer to write the expansion in terms of the cohomology spaces $\mathcal{H}^{T,\ell}_{\mathbf{d},i}$ rather than Euler characteristics as is typical in the theory of cluster algebras. The reason is that the graded vector space 
\[
\mathcal{H}^{T,\ell}=\bigoplus_{\mathbf{d},i}\mathcal{H}_{\mathbf{d},i}^{T,\ell}
\]
which categorifies the canonical basis element $\mathbb{I}(\ell)$ is closely related to the vector space of framed BPS states considered by Gaiotto, Moore, and Neitzke in the context of $\mathcal{N}=2$ field theories~\cite{GMN3}, and we suggest that Theorem~\ref{thm:introcategorification1} may provide a rigorous mathematical construction of this graded vector space.

\subsection{Formulation in terms of framed quivers}

To make contact with the theory of framed BPS states, we will now describe a reformulation of Theorem~\ref{thm:introcategorification1}. In this reformulation, $\Gr_{\mathbf{d}}^\circ(M_{T,\ell})$ is reinterpreted as a moduli space of stable representations of a framed quiver. In this introduction, we will discuss this reformulation in terms of stability conditions on a triangulated category, which is the most natural context for this result. For more details on stability conditions, we refer to~\cite{BridgelandSmith}.

As before, we consider the quiver with potential $(Q(T),W(T))$ constructed from an ideal triangulation $T$. Associated to this quiver with potential is a 3-Calabi-Yau triangulated category $\mathcal{D}(Q(T),W(T))$. Explicitly, it is defined as the subcategory of the derived category of the complete Ginzburg algebra of $(Q(T),W(T))$ consisting of modules with finite-dimensional cohomology. It follows from~\cite{KellerYang} that the categories associated to two different ideal triangulations by this construction are equivalent, and therefore we have a category $\mathcal{D}$ canonically associated to the marked bordered surface.

We consider a complex manifold $\mathcal{B}(\Sigma)$ parametrizing Bridgeland stability conditions on the category $\mathcal{D}$. In fact, the space we consider is a quotient 
\[
\mathcal{B}(\Sigma)=\Stab^*(\mathcal{D})/\Sph(\mathcal{D})
\]
where $\Stab^*(\mathcal{D})$ is a connected component of the usual manifold of stability conditions~\cite{Bridgeland07} and $\Sph(\mathcal{D})$ is a subgroup of the group of exact autoequivalences of~$\mathcal{D}$ generated by functors called \emph{spherical twists}~\cite{SeidelThomas}. A point in the space $\mathcal{B}(\Sigma)$ can be understood as a pair $(Z,\mathcal{A})$ where $\mathcal{A}$ is a full abelian subcategory of $\mathcal{D}$ and $Z:K(\mathcal{A})\rightarrow\mathbb{C}$ is a homomorphism called the \emph{central charge} mapping the Grothendieck group of~$\mathcal{A}$ into the complex numbers so that if $0\neq E\in\mathcal{A}$, then $Z([E])$ lies in the semi-closed upper half plane 
\[
\mathfrak{h}=\{r\exp(i\pi\phi):r>0 \text{ and } 0<\phi\leq1\}\subset\mathbb{C}.
\]
The space $\mathcal{B}(\Sigma)$ contains a collection of real codimension~1 subspaces called \emph{walls of the second kind}, and the complement of these walls in $\mathcal{B}(\Sigma)$ is a union of components called \emph{chambers}. If $(Z,\mathcal{A})$ lies in a chamber, then there is a corresponding ideal triangulation~$T$ such that the category $\mathcal{A}$ is equivalent to a category of finite dimensional modules
\[
\mathcal{A}\cong\Mod J(Q(T),W(T))
\]
over the Jacobian algebra $J(Q(T),W(T))$.

Suppose we are given a stability condition $\sigma=(Z,\mathcal{A})$ which does not lie on a wall of the second kind. If $T$ is the corresponding ideal triangulation, then there is a natural bijection between isomorphism classes of simple objects in~$\mathcal{A}$ and internal edges of~$T$. Therefore, we can regard the central charge map as a homomorphism 
\[
Z:\Gamma_T\rightarrow\mathbb{C}.
\]
The signed adjacency matrix of the quiver $Q=Q(T)$ defines a skew-form $\langle\cdot,\cdot\rangle$ on the lattice~$\Gamma_T$. Suppose $\ell$ is a lamination having the special property that $\langle\mathbf{e}_j,\mathbf{h}\rangle\geq0$ where $\mathbf{h}=\mathbf{h}^{T,\ell}$ is the vector appearing in Theorem~\ref{thm:introcategorification1}. Then we can modify the quiver $Q$ by drawing an additional vertex~$\infty$ and drawing $\langle\mathbf{e}_j,\mathbf{h}\rangle$ arrows starting at the vertex of $Q$ corresponding to the edge~$j$ and ending at the new vertex~$\infty$. The resulting quiver~$\widetilde{Q}$ is called a \emph{framed quiver}. The potential $W=W(T)$ determines a corresponding potential $\widetilde{W}$ for~$\widetilde{Q}$ in an obvious way, and we can extend the central charge to a map 
\[
\widetilde{Z}:\Gamma_T\oplus\mathbb{Z}\mathbf{e}_\infty\rightarrow\mathbb{C}
\]
by setting $\widetilde{Z}(\mathbf{e}_\infty)=m\zeta$ where $m\gg0$ and choosing the phase $\zeta\in U(1)$ so that $\arg\widetilde{Z}(\mathbf{e}_\infty)<\arg\widetilde{Z}(\mathbf{e}_j)$ for all~$j$. A module $M$ over $J(\widetilde{Q},\widetilde{W})$ is called \emph{stable} if we have 
\[
\arg\widetilde{Z}([N])<\arg\widetilde{Z}([M])
\]
for every proper nontrivial submodule $N\subset M$. Given a vector $\mathbf{d}=\sum d_j\mathbf{e}_j$, we write $\mathcal{M}_{\mathbf{d}}^{\mathrm{st}}(\widetilde{Q})$ for the moduli space of stable modules which, when viewed as representations of the quiver~$\widetilde{Q}$, have dimension $d_j$ at the vertex corresponding to an arc $j$ and have dimension~1 at~$\infty$.

\begin{theorem}
\label{thm:introcategorification2}
Let $\sigma$ be a stability condition with associated ideal triangulation~$T$. Let $\ell$ be a lamination such that $\langle \mathbf{e}_j,\mathbf{h}\rangle\geq0$ for all $j$ where $\mathbf{h}=\mathbf{h}^{T,\ell}$ is the vector appearing in Theorem~\ref{thm:introcategorification1}. Then there is an isomorphism of varieties $\mathcal{M}_{\mathbf{d}}^{\mathrm{st}}(\widetilde{Q})\cong\Gr_{\mathbf{d}}^\circ(M_{T,\ell})$ for any dimension vector~$\mathbf{d}$. In particular, 
\[
\mathbb{I}(\ell)=\sum_{\mathbf{d}\in\Gamma_T, i\in\mathbb{Z}}(-1)^i\dim_{\mathbb{C}}\mathcal{H}^{\sigma,\ell}_{\mathbf{d},i}\cdot X_{\mathbf{d}}
\]
where $\mathcal{H}^{\sigma,\ell}_{\mathbf{d},i}=H^i(\mathcal{M}_{\mathbf{d}-\mathbf{h}}^{\mathrm{st}}(\widetilde{Q}),\mathbb{C})$.
\end{theorem}

\subsection{Line defects and framed BPS states}

Our proposal for the categorification of Fock and Goncharov's canonical basis is closely related to the physical ideas of Gaiotto, Moore, and Neitzke~\cite{GMN3}. Their work concerns four-dimensional quantum field theories with $\mathcal{N}=2$ supersymmetry. Specifically, they consider such a field theory in the presence of a \emph{line defect}. This is a modification of the definition of the theory corresponding to a choice of one-dimensional submanifold of spacetime.

Part of the data of an $\mathcal{N}=2$ field theory is a complex manifold $\mathcal{B}$ called the \emph{Coulomb branch} which serves as a parameter space for the theory. A choice of $u\in\mathcal{B}$ corresponds in physics to a choice of \emph{vacuum}. At a generic point $u$ of the Coulomb branch, one has a lattice $\Gamma_u$ called the \emph{charge lattice}. In the presence of a line defect~$L$, the Hilbert space of the theory in the vacuum $u$ contains a distinguished $\Gamma_u$-graded subspace called the space of \emph{framed BPS states}. Information about the graded component corresponding to $\gamma\in\Gamma_u$ is encoded in an expression $\underline{\overline{\Omega}}(L,\gamma,u,y)$ called the \emph{framed protected spin character}. It is a function of a parameter $y$, and we write $\underline{\overline{\Omega}}(L,\gamma,u)$ for its specialization at $y=1$. The latter is an integer which is interpreted as counting framed BPS states of charge~$\gamma$.

An important class of $\mathcal{N}=2$ theories studied by Gaiotto, Moore, and Neitzke are the theories of \emph{class~$\mathcal{S}$}. Physically, these theories are obtained by compactifying the six-dimensional (2,0) theory on a punctured Riemann surface with extra data associated to the punctures. Mathematically, such a theory is encoded in the data of a marked bordered surface $\Sigma$. A generic point $u$ in the Coulomb branch determines a stability condition $\sigma$ and an ideal triangulation $T$ of~$\Sigma$. The charge lattice~$\Gamma_u$ can then be identified with~$\Gamma_T$. Theories of class~$\mathcal{S}$ support line defects labeled by a choice of lamination $\ell$ on the marked bordered surface and a complex number $\zeta\in\mathbb{C}^*$. It was argued in~\cite{GMN3} that for a line defect $L=L(\ell,\zeta)$ with $\arg\zeta=0$, the number $\underline{\overline{\Omega}}(L,\gamma,u)$ coincides with the coefficient $c_\gamma$ in the expansion~\eqref{eqn:expansion}.

In~\cite{CordovaNeitzke}, C\'ordova and Neitzke described another method for calculating $\underline{\overline{\Omega}}(L,\gamma,u)$. In the case of a theory of class~$\mathcal{S}$ with a line defect $L=L(\ell,\zeta)$ satisfying the condition~$\langle\mathbf{e}_j,\mathbf{h}\rangle\geq0$ for all~$j$, their method involves the moduli space $\mathcal{M}=\mathcal{M}_{\gamma-\mathbf{h}}^{\mathrm{st}}(\widetilde{Q})$ considered above. In this context, the vector $\mathbf{h}=\mathbf{h}^{T,\ell}$ is called the \emph{core charge} of the line defect and plays an important role in the physical considerations of~\cite{CordovaNeitzke}. If this moduli space is smooth, then it is a K\"ahler manifold of complex dimension~$N$, and it was proposed in~\cite{CordovaNeitzke} that the framed protected spin character is given in terms of its Hodge numbers by the expression 
\[
\underline{\overline{\Omega}}(L,\gamma,u,y)=\sum_{p,q=0}^N h^{p,q}(\mathcal{M})(-1)^{p-q}y^{2p-N}.
\]
In particular, setting $y=1$ in this expression, we see that $\underline{\overline{\Omega}}(L,\gamma,u)$ is the Euler characteristic of $\mathcal{M}$. Thus the proposal of C\'ordova and Neitzke predicts that the coefficients of the canonical functions $\mathbb{I}(\ell)$ arise as Euler characteristics of framed quiver moduli spaces. Theorem~\ref{thm:introcategorification2} confirms this prediction and shows that it is a special case of the more general categorification of canonical bases provided by Theorem~\ref{thm:introcategorification1}. It is natural to conjecture that the vector space $\mathcal{H}^{T,\ell}$ defined by this theorem is closely related to the space of framed BPS~states mentioned above.

This connection between categorified basis elements and framed BPS states builds on many earlier results in mathematics and physics. In~\cite{Cirafici13}, the relationship between framed quivers and framed BPS states was studied in theories of class~$\mathcal{S}$. In~\cite{CDMMS,Williams,Cirafici18,CiraficiDelZotto}, framed quiver moduli spaces were used to compute framed BPS indices. In particular, the reference~\cite{Williams} by Williams shares many similarities with the present paper. It discusses the relationship with cluster characters and quiver Grassmannians for another class of field theories defined by a quiver with potential. The idea that framed BPS indices can be computed by taking cluster characters of string and band objects was also suggested in~\cite{CaorsiCecotti}.

\subsection{Further directions}

The results of this paper suggest several interesting questions for future research. Below we describe some of these problems and review the relevant literature.

\subsubsection{Generalization to arbitrary surfaces}
\label{sec:GeneralizationToArbitrarySurfaces}

One obvious problem is to generalize the construction of the present paper to the case where $\Sigma$ is an arbitrary marked bordered surface, possibly with punctures. In the case of a punctured surface, the quiver with potential associated to an ideal triangulation was defined in~\cite{L1}. The cluster category associated to such a surface is understood in this case, provided the surface has at least one boundary component~\cite{QiuZhou}. Modules associated to arcs were constructed in~\cite{L2} in the case where the ideal triangulation has no self-folded triangles, and more generally in~\cite{Dominguez}. In the latter two papers, difficult calculations are needed to prove that the modules constructed are annihilated by the Jacobian relations. On the other hand, there does not appear to be any discussion in the existing literature of modules associated to closed loops.

\subsubsection{Categorification of theta functions}

In~\cite{GHKK}, Gross, Hacking, Keel, and Kontsevich used ideas from mirror symmetry to construct canonical bases in a more general setting than the one considered here. In~\cite{GS}, Goncharov and Shen showed that their results can be applied to cluster varieties arising from marked bordered surfaces. The elements of these canonical bases are called \emph{theta functions} and are believed to coincide with the functions $\mathbb{I}(\ell)$ of Fock and Goncharov although this is not known at present. In~\cite{Bridgeland16}, Bridgeland showed showed how to categorify some of the theta functions using moduli spaces of framed quiver representations. The result of~\cite{Bridgeland16} imposes a condition similar to the condition $\langle\mathbf{e}_j,\mathbf{h}\rangle\geq0$ in the statement of Theorem~\ref{thm:introcategorification2}. In view of the results of the present paper, it seems likely that the full basis of theta functions could be categrorified by replacing the framed quiver moduli spaces of~\cite{Bridgeland16} by quiver Grassmannians or constructible subsets of these varieties.

\subsubsection{Categorification of quantized canonical bases}

As its name suggests, the cluster Poisson variety $\mathcal{X}$ has a natural Poisson structure and can be canonically quantized. In other words, there exists a family of noncommutative algebras~$\mathcal{O}_q(\mathcal{X})$, depending on a parameter~$q\in\mathbb{C}^*$, such that $\mathcal{O}_q(\mathcal{X})$ coincides with the algebra of regular functions on $\mathcal{X}$ in the classical $q=1$ limit. In~\cite{AllegrettiKim}, the canonical basis construction of Fock and Goncharov was generalized to this quantum setting in the case where $\Sigma$ is a punctured surface without boundary. In other words, it was shown that there exist elements $\mathbb{I}^q(\ell)\in\mathcal{O}_q(\mathcal{X})$ which coincide with the regular functions $\mathbb{I}(\ell)$ defined by Fock and Goncharov in the classical limit. This result was later extended to the case where $\Sigma$ is an unpunctured disk with marked points on its boundary~\cite{Allegretti16}, and the positivity properties of the quantized basis elements were further discussed in~\cite{CKKO}. In addition to these results describing the quantization of the canonical basis, many of the results used in the present paper to categorify the canonical basis also have quantum analogs. For example, the result that we use to categorify cluster variables has a quantum analog discussed from different points of view in~\cite{Rupel,QinKeller,Efimov}. It would be interesting to use these results to construct a categorification of the quantized basis elements $\mathbb{I}^q(\ell)$. Such a construction would provide a mathematical approach to the framed protected spin characters in theories of class~$\mathcal{S}$. For partial results in this direction, see~\cite{CarnakciLampe}.

\subsubsection{Monoidal categorification}

Theorem~\ref{thm:introcategorification1} shows that the vector space $\mathcal{H}^\ell=\mathcal{H}^{T,\ell}$ provides a categorification of the canonical function $\mathbb{I}(\ell)$ in the sense that its graded dimension gives the expansion of this function in cluster coordinates. It would be interesting to promote this to a monoidal categorification of the algebra $\mathcal{O}(\mathcal{X})$. For any two laminations $\ell$ and $\ell'$, it should be possible to find vector spaces $\mathcal{N}^{\ell,\ell'}_{\ell''}$ such that 
\[
\mathcal{H}^\ell\otimes\mathcal{H}^{\ell'}=\bigoplus_{\ell''}\mathcal{N}^{\ell,\ell'}_{\ell''}\otimes\mathcal{H}^{\ell''}
\]
where the sum is over all laminations~$\ell''$. By taking graded dimensions, we should recover the expansion of the product $\mathbb{I}(\ell)\cdot\mathbb{I}(\ell')$ as a linear combination of canonical basis elements, and the vector spaces $\mathcal{N}^{\ell,\ell'}_{\ell''}$ should categorify the structure constants in this expansion. This decomposition should be closely related to the categorified OPE algebra defined physically in Section~2.4 of~\cite{CordovaNeitzke} and to Question~1.5 in~\cite{Thurston}. Recently, Cautis and Williams constructed a similar monoidal categorification for a different choice of quiver with potential and discussed the relation to line defects~\cite{CautisWilliams}.

\subsection{Structure of the paper}

In Section~\ref{sec:ClusterVarietiesFromSurfaces}, we review the basic material on marked bordered surfaces and ideal triangulations from~\cite{FST} and the definition of the associated quivers from~\cite{L1}. We discuss moduli spaces of local systems and cluster coordinates following~\cite{FG1}. We then define cluster varieties and review some results we need from the theory of cluster algebras.

In Section~\ref{sec:TheCanonicalBasisConstruction}, we recall the notion of an integral lamination from~\cite{FG1,FG2}. We explain a method for calculating the monodromy of a framed local system around a closed loop, and we use this method to assign a canonical function to any lamination consisting of a single closed loop. We also associate a function to any lamination consisting of curves that end on the boundary of a surface, and we relate this function to $F$-polynomials in the theory of cluster algebras. Finally, we give the general definition of the canonical function associated to a lamination, and we prove Theorem~\ref{thm:introcanonicalbasis}.

In Section~\ref{sec:RepresentationsOfQuivers}, we review the basic representation theory of quivers with potential. We define the notions of strings and bands following~\cite{BrustleZhang} and the quiver Grassmannian of a representation. We then define band modules and relate them to the canonical functions associated to loops. We define string modules and explain the relationship with $F$-polynomials. We conclude with the proof of Theorem~\ref{thm:introcategorification1}.

In Section~\ref{sec:FramedQuiverModuli}, we discuss stability conditions on the abelian category of modules over the Jacobian algebra of a quiver with potential. We describe moduli spaces parametrizing cocyclic modules over the Jacobian algebra. We then establish the relationship to quiver Grassmannians and prove Theorem~\ref{thm:introcategorification2}.

\section{Cluster varieties from surfaces}
\label{sec:ClusterVarietiesFromSurfaces}

\subsection{Combinatorics of triangulations}

We begin by reviewing the basic material on triangulated surfaces. The treatment in this section is based on~\cite{FST}.

\begin{definition}
A \emph{marked bordered surface} (without punctures) is a pair $(\mathbb{S},\mathbb{M})$ where $\mathbb{S}$ is a compact, connected, oriented, smooth surface with boundary and $\mathbb{M}$ is a nonempty finite set of marked points on the boundary of~$\mathbb{S}$ such that each boundary component contains at least one marked point.
\end{definition}

In this paper, we will denote by $\Sigma=(\mathbb{S},\mathbb{M})$ a marked bordered surface.

\begin{definition}
An \emph{arc} on $\Sigma$ is a smooth path $\gamma$ in~$\mathbb{S}$ connecting points of~$\mathbb{M}$ whose interior lies in the interior of~$\mathbb{S}$ and which has no self-intersections in its interior. We also require that $\gamma$ is not homotopic, relative to its endpoints, to a single point or to a path in~$\partial\mathbb{S}$ whose interior contains no marked points. A path that connects two marked points and lies entirely on the boundary of~$\mathbb{S}$ without passing through a third marked point is called a \emph{boundary segment}.
\end{definition}

\begin{definition}
Two arcs are considered to be equivalent if they are related by a homotopy through arcs. Two arcs are \emph{compatible} if there exist arcs in their respective equivalence classes which do not intersect in the interior of~$\mathbb{S}$. An \emph{ideal triangulation} of $\Sigma$ is a maximal collection of pairwise compatible arcs on~$\Sigma$.
\end{definition}

When talking about an ideal triangulation of $\Sigma$, we will always fix a collection of representatives for its arcs so that no two arcs intersect in the interior of~$\mathbb{S}$.

\begin{definition}
A \emph{triangle} of an ideal triangulation $T$ is defined to be the closure in~$\mathbb{S}$ of a connected component of the complement of all arcs of~$T$. By an \emph{edge} of an ideal triangulation, we mean an arc of the triangulation or a boundary segment.
\end{definition}

There is an integer matrix encoding the combinatorics of an ideal triangulation~$T$ of a marked bordered surface. To define this matrix, we consider for each triangle $t$ and each pair of edges $i$,~$j$ of $T$, a number $\varepsilon_{ij}^t$ defined by the following rules:
\begin{enumerate}
\item $\varepsilon_{ij}^t=+1$ if $i$ and $j$ are sides of $t$ with $j$ following $i$ in the counterclockwise order defined by the orientation.
\item $\varepsilon_{ij}^t=-1$ if the same holds with the clockwise order.
\item $\varepsilon_{ij}^t=0$ otherwise.
\end{enumerate}
We then have the following definition.

\begin{definition}
The \emph{signed adjacency matrix} associated to $T$ is the matrix with rows and columns indexed by the edges of $T$ whose $(i,j)$ entry $\varepsilon_{ij}=\varepsilon_{ij}^T$ is defined as 
\[
\varepsilon_{ij}=\sum_t\varepsilon_{ij}^t,
\]
where the sum runs over all triangles of~$T$.
\end{definition}

The following is an elementary move for passing between different ideal triangulations.

\begin{definition}
A \emph{flip} of an ideal triangulation $T$ at an arc $\gamma$ of $T$ is the transformation that removes $\gamma$ and replaces it by the unique different arc that, together with the remaining arcs, forms a new ideal triangulation (see Figure~\ref{fig:flip}).
\end{definition}

\begin{figure}[ht]
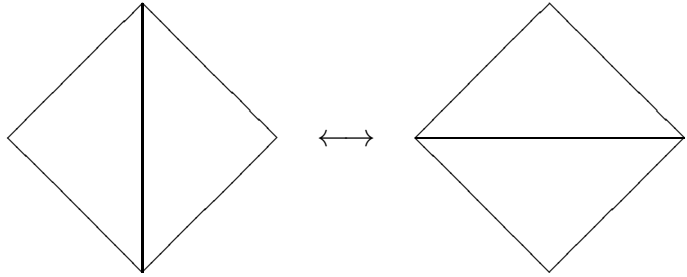

\begin{center}
\[
\xy /l1.5pc/:
{\xypolygon4"A"{~:{(2,2):}}};
{\xypolygon4"B"{~:{(2.5,0):}~>{}}};
{\xypolygon4"C"{~:{(0.8,0.8):}~>{}}};
{"A1"\PATH~={**@{-}}'"A3"};
\endxy
\quad
\longleftrightarrow
\quad
\xy /l1.5pc/:
{\xypolygon4"A"{~:{(2,2):}}};
{\xypolygon4"B"{~:{(2.5,0):}~>{}}};
{\xypolygon4"C"{~:{(0.8,0.8):}~>{}}};
{"A2"\PATH~={**@{-}}'"A4"};
\endxy
\]
\end{center}
\caption{A flip of a triangulation.\label{fig:flip}}
\end{figure}

The following fact is well known.

\begin{proposition}[\cite{FST}, Proposition~3.8]
\label{prop:flipelementary}
Any two ideal triangulations of a marked bordered surface are related by a sequence of flips.
\end{proposition}

In this paper, we will be interested in a certain quiver associated to an ideal triangulation of a marked bordered surface.

\begin{definition}
A \emph{quiver} is a directed graph. It consists of a finite set $Q_0$ (the set of \emph{vertices}), a finite set $Q_1$ (the set of \emph{arrows}), and maps $s:Q_1\rightarrow Q_0$ and $t:Q_1\rightarrow Q_0$ taking an arrow to its \emph{source} and \emph{target}, respectively. We typically display an arrow diagrammatically as 
\[
\xymatrix{
s(a) \ar[rr]^{a} & & t(a).
}
\]
A \emph{loop} in a quiver $Q$ is an arrow $a$ whose source and target coincide. A \emph{2-cycle} is a pair of distinct arrows $a$ and $b$ such that the target of $a$ is the source of $b$ and vice versa. A quiver is said to be \emph{2-acyclic} if it has no loops or 2-cycles.
\end{definition}

Given an ideal triangulation $T$ of the marked bordered surface $\Sigma$, we get a quiver $Q(T)$ in a natural way. This quiver has as its vertex set the set of all arcs of~$T$, and there are $\varepsilon_{ij}$ arrows from~$i$ to~$j$ whenever $\varepsilon_{ij}>0$. Note that this quiver is 2-acyclic since the signed adjacency matrix $\varepsilon_{ij}$ is skew-symmetric. Figure~\ref{fig:quivertriangulation} in the introduction shows an example of $Q(T)$ in the case where $T$ is an ideal triangulation of a disk with eight marked points on its boundary.

The following is a natural operation on quivers.

\begin{definition}
Let $k$ be a vertex of a 2-acyclic quiver $Q$. Then we define a new quiver $\mu_k(Q)$, called the quiver obtained by \emph{mutation} in the direction~$k$, as follows.
\begin{enumerate}
\item For each pair of arrows $i\rightarrow k\rightarrow j$, add a new arrow $i\rightarrow j$.
\item Reverse all arrows incident to~$k$.
\item Remove the arrows from a maximal set of pairwise disjoint 2-cycles.
\end{enumerate}
\end{definition}

The last item in this definition means, for example, that we replace the diagram $\xymatrix{i \ar@< 4pt>[r] \ar@< -4pt>[r] & j \ar[l]}$ by $\xymatrix{i \ar[r] & j}$. Thus if $Q$ is a 2-acyclic quiver, then $\mu_k(Q)$ will again be 2-acyclic. The operation of mutation is an involution on such quivers in the sense that $\mu_k\circ\mu_k(Q)\cong Q$.

The following is easy to check using the definition of mutation.

\begin{proposition}
\label{prop:flipmutation}
If $T$ and $T'$ are ideal triangulations of $\Sigma$ such that $T'$ is obtained from~$T$ by a flip at some arc~$k$, then $\mu_k(Q(T))=Q(T')$.
\end{proposition}

\subsection{Moduli spaces of local systems}

We now define the moduli space of framed $PGL_2$-local systems. Let $\Sigma=(\mathbb{S},\mathbb{M})$ be a marked bordered surface without punctures, and let $\mathcal{L}$ be a $PGL_2$-local system on~$\mathbb{S}$, that is, a principal $PGL_2$-bundle equipped with a flat connection. Since the group $PGL_2$ has a natural left action on~$\mathbb{P}^1$, we can form the associated bundle 
\[
\mathcal{L}_{\mathbb{P}^1}\coloneqq\mathcal{L}\times_{PGL_2}\mathbb{P}^1.
\]
For the next definition, let us fix a point $x_i$ in the interior of every boundary segment~$i$ and denote by $\partial^\circ\mathbb{S}=\partial\mathbb{S}\setminus\{x_i\}$ the boundary of $\mathbb{S}$ punctured at these points.

\begin{definition}[\cite{FG1}]
A \emph{framing} for a $PGL_2$-local system $\mathcal{L}$ on~$\mathbb{S}$ is defined as a flat section of the restriction of~$\mathcal{L}_{\mathbb{P}^1}$ to the punctured boundary $\partial^\circ\mathbb{S}$. A \emph{framed $PGL_2$-local system} on~$\Sigma$ is a $PGL_2$-local system on~$\mathbb{S}$ together with a framing.
\end{definition}

An isomorphism of framed local systems $(\mathcal{L}_1,s_1)$ and $(\mathcal{L}_2,s_2)$ is an isomorphism $\mathcal{L}_1\rightarrow\mathcal{L}_2$ of the underlying local systems such that the induced map on $\mathbb{P}^1$-bundles sends the flat section $s_1$ to $s_2$.

\begin{definition}
For any marked bordered surface $\Sigma$, we write $\mathcal{X}_{\Sigma,PGL_2}$ for the moduli stack parametrizing isomorphism classes of framed $PGL_2$-local systems on~$\Sigma$.
\end{definition}

We will also consider a moduli space parametrizing twisted $SL_2$-local systems equipped with a decoration. Let $T'\mathbb{S}$ be the punctured tangent bundle of~$\mathbb{S}$, that is, the tangent bundle with the zero section removed. For any point $y\in\mathbb{S}$, we have $T_y\mathbb{S}\cong\mathbb{R}^2$. Thus $T_y'\mathbb{S}=T_y\mathbb{S}-\{0\}$ has fundamental group 
\[
\pi_1(T_y'\mathbb{S},x)\cong\mathbb{Z}
\]
for any choice of basepoint $x\in T_y'\mathbb{S}$. Let $\sigma_{\mathbb{S}}$ denote a generator of this fundamental group. It is well defined up to a sign. By abuse of notation, we will also write $\sigma_{\mathbb{S}}$ for the image of this generator under the inclusion $\pi_1(T_y'\mathbb{S},x)\hookrightarrow \pi_1(T'\mathbb{S},x)$. The group $\pi_1(T'\mathbb{S},x)$ fits into a short exact sequence 
\[
\xymatrix{ 
1 \ar[r] & \mathbb{Z} \ar[r] & \pi_1(T'\mathbb{S},x) \ar[r] & \pi_1(\mathbb{S},y) \ar[r] & 1
}
\]
where the group $\mathbb{Z}$ is identified with the central subgroup of $\pi_1(T'\mathbb{S},x)$ generated by~$\sigma_{\mathbb{S}}$.

\begin{definition}
A \emph{twisted $SL_2$-local system} $\mathcal{L}$ on $\mathbb{S}$ is an $SL_2$-local system on the punctured tangent bundle $T'\mathbb{S}$ with monodromy $-1$ around $\sigma_{\mathbb{S}}$.
\end{definition}

Let $\widehat{\mathcal{L}}$ be a twisted $SL_2$-local system on~$\mathbb{S}$. The space $\mathbb{A}^2\setminus\{0\}$ is identified with the \emph{affine flag variety} for the group~$SL_2$ (see \cite{FG1} for a more general discussion). Since the group $SL_2$ has a natural left action on $\mathbb{A}^2\setminus\{0\}$, we have the associated bundle 
\[
\widehat{\mathcal{L}}_{\mathbb{A}^2\setminus\{0\}}=\widehat{\mathcal{L}}\times_{SL_2}(\mathbb{A}^2\setminus\{0\}).
\]
The punctured tangent bundle $T'\mathbb{S}$ has a section over the punctured boundary $\partial^\circ\mathbb{S}$ defined by outward pointing tangent vectors at points of~$\partial^\circ\mathbb{S}$. In the following definition, we will denote this section by $\widehat{\partial^\circ\mathbb{S}}$.

\begin{definition}
A \emph{decoration} of a twisted $SL_2$-local system $\widehat{\mathcal{L}}$ is a flat section of the restriction of $\widehat{\mathcal{L}}_{\mathbb{A}^2\setminus\{0\}}$ to the lifted punctured boundary $\widehat{\partial^\circ\mathbb{S}}$. A \emph{decorated twisted $SL_2$-local system} on $\Sigma$ is a twisted $SL_2$-local system on $T'\mathbb{S}$ together with a decoration.
\end{definition}

As in the case of framed local systems, an isomorphism of decorated twisted $SL_2$-local systems is an isomorphism of the underlying local systems that preserves the decorations.

\begin{definition}
For any marked bordered surface $\Sigma$, we write $\mathcal{A}_{\Sigma,SL_2}$ for the moduli stack parametrizing isomorphism classes of decorated twisted $SL_2$-local systems on~$\Sigma$.
\end{definition}

There is a natural map 
\begin{align}
\label{eqn:pmap}
p:\mathcal{A}_{\Sigma,SL_2}\rightarrow\mathcal{X}_{\Sigma,PGL_2}
\end{align}
sending a decorated twisted $SL_2$-local system $(\widehat{\mathcal{L}},\widehat{s})$ to a framed $PGL_2$-local system $p(\widehat{\mathcal{L}},\widehat{s})=(\mathcal{L},s)$ defined as follows. To define the local system $\mathcal{L}$, we simply take the pushforward of~$\widehat{\mathcal{L}}$ which gives a $PGL_2$-local system on~$\mathbb{S}$. By definition, there is a map $\mathbb{A}^2\setminus\{0\}\rightarrow\mathbb{P}^1$ sending a nonzero vector to the line containing it. It follows that there is a map of sheaves from the pushforward of  $\widehat{\mathcal{L}}_{\mathbb{A}^2\setminus\{0\}}$ to $\mathcal{L}_{\mathbb{P}^1}$, and we define $s$ to be the image of~$\widehat{s}$ under this map.

We now discuss certain functions on the moduli space $\mathcal{A}_{\Sigma,SL_2}$. To define them, let $i$ be an arc or boundary segment on~$\Sigma$, with endpoints $p_1$,~$p_2\in\mathbb{M}$. The section $\widehat{\partial^\circ\mathbb{S}}$ described above determines an outward pointing tangent vector $u_k$ at the point $p_k$ for $k=1,2$, and there is a canonical homotopy class of paths in $T'\mathbb{S}$ connecting $u_2$ to~$u_1$. Indeed, let $\mathbb{D}$ be the disk defined as a small neighborhood of the edge $i$ having the points $p_1$ and $p_2$ on its boundary, and let $\gamma$ be the path in $T'\mathbb{S}$ defined by dragging the vector $u_2$ along the boundary of $\mathbb{D}$ in the counterclockwise direction, always pointing out, until it coincides with~$u_1$.

Let $\Bbbk$ be a field and $(\widehat{\mathcal{L}},\widehat{s})$ a general $\Bbbk$-point of the moduli space $\mathcal{A}_{\Sigma,SL_2}$. Then the decoration $\widehat{s}$ determines a vector in the fiber over $u_k$ for $k=1,2$. Using the flat connection, we can parallel transport along $\gamma$ to get a pair of vectors $v_1$,~$v_2$ in the fiber over some chosen point.

\begin{definition}
\label{def:Acoordinate}
The \emph{cluster coordinate} associated to $i$ is given by 
\[
A_i=\omega(v_1\wedge v_2)
\]
where $\omega$ is a fixed volume form preserved by the $SL_2$ action on the fiber of~$\widehat{\mathcal{L}}$.
\end{definition}

Since we consider twisted local systems, it is possible to show that this definition is independent of the labeling of the points~$p_1$ and~$p_2$.

In particular, if $T$ is an ideal triangulation of~$\Sigma$, then the above construction defines a collection of rational functions $A_i:\mathcal{A}_{\Sigma,SL_2}\dashrightarrow\mathbb{G}_m$ indexed by the edges of~$T$. In fact, we have the following statement.

\begin{proposition}[\cite{FG1}]
\label{prop:Abirational}
Let $\Sigma$ be a marked bordered surface. Then for any ideal triangulation~$T$ of~$\Sigma$, the functions $A_i$ provide a birational map 
\[
\mathcal{A}_{\Sigma,SL_2}\dashrightarrow(\mathbb{G}_m)^I
\]
where $I$ is the set of edges of $T$.
\end{proposition}

To describe the transition maps relating these coordinate charts for different ideal triangulations, it suffices to consider two ideal triangulations $T$ and~$T'$ related by a flip at an arc~$k$. In this case, the set $I$ of edges of~$T$ is naturally in bijection with the set $I'$ of edges of~$T'$, and thus we can use the construction described above to associate a coordinate $A_i'$ to each $i\in I'=I$.

\begin{proposition}
The coordinates $A_i'$ are given in terms of the coordinates $A_i$~($i\in I$) by the formula 
\begin{align}
\label{eqn:Atransformation}
A_i' = 
\begin{cases}
\frac{\prod_{\varepsilon_{kj}>0}A_j^{\varepsilon_{kj}}+\prod_{\varepsilon_{kj}<0}A_j^{-\varepsilon_{kj}}}{A_k} & \text{if $i=k$} \\
A_i & \text{if $i\neq k$}.
\end{cases}
\end{align}
\end{proposition}

We can now easily define coordinates on the moduli space $\mathcal{X}_{\Sigma,PGL_2}$ of framed local systems. As before, we let $T$ be an ideal triangulation of~$\Sigma$ and we choose a general $\Bbbk$-point $(\mathcal{L},s)$ of the moduli space $\mathcal{X}_{\Sigma,PGL_2}$. If $j$ is any arc of~$T$, then we consider the quadrilateral $q$ formed by the two triangles that share the edge~$j$. This quadrilateral $q$ is naturally a marked bordered surface, and we get a framed local system on~$q$ by restriction of $(\mathcal{L},s)$. Choose a decorated twisted local system on~$q$ that is mapped to this framed local system by~$p$, and let $A_i$ be its coordinates.

\begin{definition}
The \emph{cluster coordinate} associated to the arc $j$ is given by 
\[
X_j=\prod_iA_i^{\varepsilon_{ji}}.
\]
\end{definition}

One can check that this definition is independent of the choice of decorated twisted local system. 

In this way, we define a collection of rational functions $X_j:\mathcal{X}_{\Sigma,PGL_2}\dashrightarrow\mathbb{G}_m$ indexed by the arcs of an ideal triangulation. In fact, we have the following statement.

\begin{proposition}[\cite{FG1}]
\label{prop:Xbirational}
Let $\Sigma$ be a marked bordered surface. Then for any ideal triangulation~$T$ of~$\Sigma$, the functions $X_j$ provide a birational map 
\[
\mathcal{X}_{\Sigma,PGL_2}\dashrightarrow(\mathbb{G}_m)^J
\]
where $J$ is the set of arcs of $T$.
\end{proposition}

Consider two ideal triangulations $T$ and~$T'$ related by a flip at an arc~$k$. Then the set~$J$ of arcs of~$T$ is naturally in bijection with the set $J'$ of arcs of~$T'$, and we can associate a coordinate $X_j'$ to each $j\in J'=J$.

\begin{proposition}
The coordinates $X_j'$ are given in terms of the coordinates $X_j$~($j\in J$) by the formula 
\[
X_j' = 
\begin{cases}
X_k^{-1} & \text{if $j=k$} \\
X_j{(1+X_k^{-\sgn(\varepsilon_{jk})})}^{-\varepsilon_{jk}} & \text{if $j\neq k$}.
\end{cases}
\]
\end{proposition}

\subsection{Construction of cluster varieties}

We now define cluster varieties associated to a marked bordered surface. In addition to the cluster Poisson variety~$\mathcal{X}$ described in the introduction, it will useful to consider another cluster variety denoted~$\mathcal{A}$ in the works of Fock and Goncharov. The latter is closely related to Fomin and Zelevinsky's notion of cluster algebra~\cite{FZIV}.

\begin{definition}
An \emph{ice quiver} is a 2-acyclic quiver together with a distinguished collection of vertices called \emph{frozen vertices}. A vertex which is not frozen is called \emph{mutable}.
\end{definition}

For any ideal triangulation $T$ of the marked bordered surface $\Sigma$, we get an ice quiver whose vertices are the edges of~$T$ with $\varepsilon_{ij}$ arrows from~$i$ to~$j$ whenever $\varepsilon_{ij}>0$. The frozen vertices of this quiver are the boundary segments of~$\Sigma$. Figure~\ref{fig:icequiver} shows an example where $\Sigma$ is a disk with six marked points on its boundary.

\begin{figure}[ht]
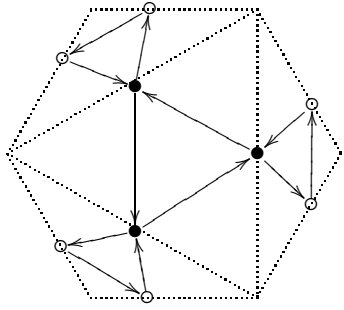
 \begin{center}
\[
\xy /l1.75pc/:
{\xypolygon6"A"{~:{(3,0):}~>{.}}},
{\xypolygon9"B"{~:{(2.3,1.3):}~>{}}},
{"A1"\PATH~={**@{.}}'"A3"'},
{"A1"\PATH~={**@{.}}'"A5"'},
{"A3"\PATH~={**@{.}}'"A5"'},
(-0.5,0)*{\bullet}="a", 
(1.7,-1.2)*{\bullet}="b",
(1.7,1.4)*{\bullet}="c",
"B1"*{\circ}="d", 
"B2"*{\circ}="e", 
"B4"*{\circ}="f", 
"B5"*{\circ}="g", 
"B7"*{\circ}="h", 
"B8"*{\circ}="i", 
\xygraph{
"a":"b",
"b":"c",
"c":"a",
"c":"d",
"d":"e",
"e":"c",
"a":"f",
"f":"g",
"g":"a",
"b":"h",
"h":"i",
"i":"b",
}
\endxy
\]
\caption{An ice quiver with mutable vertices $\bullet$ and frozen vertices $\circ$.\label{fig:icequiver}}
\end{center} \end{figure}

Suppose $Q$ is an ice quiver. Let $I$ denote the set of all vertices of~$Q$ and $J\subseteq I$ the subset of mutable vertices. Let $\mathbb{T}_n$ denote an $n$-regular tree where $n=|J|$. We can label the edges of~$\mathbb{T}_n$ by elements of $J$ in such a way that the $n$ edges emanating from any vertex have distinct labels. Choose a vertex $t_0$ of $\mathbb{T}_n$ and associate the quiver $Q=Q(t_0)$ to this vertex. We associate a quiver $Q(t)$ to every other vertex $t$ of $\mathbb{T}_n$ in such a way that if two vertices are connected by an edge labeled $k$, then the quivers associated to these vertices are related by a mutation in the direction~$k$.

Now for any vertex $t$ of the tree $\mathbb{T}_n$, we have the algebraic tori 
\[
\mathcal{A}_t=(\mathbb{G}_m)^I, \quad \mathcal{X}_t=(\mathbb{G}_m)^J.
\]
We also have a matrix $\varepsilon_{ij}=\varepsilon_{ij}^{(t)}$ defined by 
\begin{align}
\label{eqn:adjacency}
\varepsilon_{ij}=|\{\text{arrows from $i$ to $j$ in $Q(t)$}\}|-|\{\text{arrows from $j$ to $i$ in $Q(t)$}\}|.
\end{align}
If $t$ and $t'$ are vertices of $\mathbb{T}_n$ connected by an edge labeled $k$, then there are birational maps 
\[
\mu_k:\mathcal{X}_t\dashrightarrow\mathcal{X}_{t'}, \quad \mu_k:\mathcal{A}_t\dashrightarrow\mathcal{A}_{t'}.
\]
Abusing notation, we denote them both by~$\mu_k$. If $A_i$ and $X_j$ are the natural coordinates on~$\mathcal{A}_t$ and $\mathcal{X}_t$, respectively, and we write $A_i'$ and $X_j'$ for the similar coordinates on $\mathcal{A}_{t'}$ and~$\mathcal{X}_{t'}$, then these maps are defined by 
\[
\mu_k^*(A_i') = 
\begin{cases}
\frac{\prod_{\varepsilon_{kj}>0}A_j^{\varepsilon_{kj}}+\prod_{\varepsilon_{kj}<0}A_j^{-\varepsilon_{kj}}}{A_k} & \text{if $i=k$} \\
A_i & \text{if $i\neq k$}
\end{cases}
\]
and
\[
\mu_k^*(X_j') = 
\begin{cases}
X_k^{-1} & \text{if $j=k$} \\
X_j{(1+X_k^{-\sgn(\varepsilon_{jk})})}^{-\varepsilon_{jk}} & \text{if $j\neq k$}.
\end{cases}
\]
More generally, if $t$ and $t'$ are any vertices of $\mathbb{T}_n$, there is a unique simple path from $t$ to~$t'$. By composing the maps $\mu_k$ in order along the path connecting $t$ and $t'$, we obtain birational maps $\mathcal{A}_t\dashrightarrow\mathcal{A}_{t'}$ and $\mathcal{X}_t\dashrightarrow\mathcal{X}_{t'}$. 

\begin{lemma}[\cite{GHK}, Proposition~2.4]
Let $\{Z_i\}$ be a collection of integral separated schemes of finite type over~$\mathbb{Q}$ and suppose we have birational maps $f_{ij}:Z_i\dashrightarrow Z_j$ for all $i$, $j$ such that $f_{ii}$ is the identity and $f_{jk}\circ f_{ij}=f_{ik}$ as rational maps. Let $U_{ij}$ be the largest open subset of $Z_i$ such that $f_{ij}:U_{ij}\rightarrow f_{ij}(U_{ij})$ is an isomorphism. Then there is a scheme obtained by gluing the $Z_i$ along the open sets $U_{ij}$ using the maps $f_{ij}$.
\end{lemma}

Using this lemma, we can glue the tori defined above to get a scheme.

\begin{definition}
Let $Q$ be an ice quiver. Then the \emph{cluster $K_2$-variety} associated to~$Q$ is the scheme obtained by gluing the tori $\mathcal{A}_t$ for all vertices $t$ of~$\mathbb{T}_n$ using the above birational maps. The \emph{cluster Poisson variety} is the scheme obtained by gluing the tori $\mathcal{X}_t$ for all vertices $t$ of~$\mathbb{T}_n$ using the above maps.
\end{definition}

As the names suggest, the cluster $K_2$-variety is equipped with a canonical class in $K_2$ of its function field, while the cluster Poisson variety has a canonical Poisson structure.

By Propositions~\ref{prop:flipmutation} and~\ref{prop:flipelementary}, we see that there is a cluster $K_2$-variety $\mathcal{A}$ and a cluster Poisson variety~$\mathcal{X}$ canonically associated to a marked bordered surface $\Sigma$. Note that the birational maps used to glue the tori in the definition of these cluster varieties are the transition maps for coordinates on the moduli space $\mathcal{A}_{\Sigma,SL_2}$ and $\mathcal{X}_{\Sigma,PGL_2}$, respectively. Combining this observation with Propositions~\ref{prop:Abirational} and~\ref{prop:Xbirational}, we obtain the following statement.

\begin{proposition}
There exist canonical birational maps 
\[
\mathcal{A}_{\Sigma,SL_2}\dashrightarrow\mathcal{A}, \quad \mathcal{X}_{\Sigma,PGL_2}\dashrightarrow\mathcal{X}.
\]
\end{proposition}

\subsection{Cluster algebras and $F$-polynomials}

In this subsection, we briefly recall some results from the theory of cluster algebras. This material will be used later to study the canonical basis. Throughout this subsection, we fix a positive integer $m$ and write $\mathcal{F}$ for a field isomorphic to the field of rational functions in $m$ independent variables with coefficients in~$\mathbb{Q}$.

\begin{definition}
By a \emph{seed} we mean an ice quiver $Q$ with $m$ vertices together with an $m$-tuple $(A_i)_{i\in Q_0}$ of elements of~$\mathcal{F}$, such that the $A_i$ are algebraically independent over~$\mathbb{Q}$ and $\mathcal{F}=\mathbb{Q}(A_i:i\in Q_0)$. The $A_i$ are called \emph{cluster variables} and the tuple $(A_i)_{i\in Q_0}$ is called a \emph{cluster}.
\end{definition}

Typically, we number the vertices of $Q$ from~1 to~$m$ so that a cluster can be written as an $m$-tuple $(A_1,\dots,A_m)$. We choose this numbering in such a way that the mutable vertices of~$Q$ correspond to $1,\dots,n$.

\begin{definition}
\label{def:seedmutation}
Let $(Q,(A_1,\dots,A_m))$ be a seed, and let $k$ be a mutable vertex of~$Q$. Then we define a new seed $(Q',(A_1',\dots,A_m'))$ called the seed obtained by \emph{mutation} in the direction~$k$. Here $Q'=\mu_k(Q)$ and the cluster variables $A_1',\dots,A_m'$ are given by the formula~\eqref{eqn:Atransformation} where the matrix $\varepsilon_{ij}$ in this formula is defined by replacing $Q(t)$ with $Q$ in~\eqref{eqn:adjacency}.
\end{definition}

As in the definition of the cluster varieties, we write $\mathbb{T}_n$ for the $n$-regular tree with edges labeled by the numbers $1,\dots,n$ in such a way that the $n$ edges emanating from any vertex have distinct labels. Let us assign a seed $\mathbf{s}_t$ to every vertex $t\in\mathbb{T}_n$ so that if $t$ and $t'$ are vertices connected by an edge labeled~$k$, then $\mathbf{s}_{t'}$ is obtained from $\mathbf{s}_t$ by a mutation in the direction~$k$. We use the notation 
\[
\mathbf{s}_t=(\varepsilon_t,(A_{1;t},\dots,A_{m;t})), \quad \varepsilon_t=(\varepsilon_{ij}^t)
\]
for the data of these seeds.

\begin{definition}
Consider the set 
\[
\mathcal{S}=\{A_{l;t}:1\leq l\leq n,t\in\mathbb{T}_n\}\cup\{A_{l;t}^{\pm1}:n+1\leq l\leq m,t\in\mathbb{T}_n\}.
\]
The \emph{cluster algebra} is defined as the $\mathbb{Z}$-subalgebra of~$\mathcal{F}$ generated by this set~$\mathcal{S}$. For each cluster $(A_{1;t},\dots,A_{m;t})$, we have the subring 
\[
\mathbb{Z}[A_{1;t}^{\pm1},\dots,A_{m;t}^{\pm1}]\subset\mathcal{F}
\]
of Laurent polynomials in the variables $A_{1;t},\dots,A_{m;t}$, and we define the \emph{upper cluster algebra} to be the intersection of these subrings for all $t\in\mathbb{T}_m$.
\end{definition}

In particular, we can take the quiver $Q$ in an initial seed to be the ice quiver associated to an ideal triangulation of a marked bordered surface $\Sigma=(\mathbb{S},\mathbb{M})$. Then we get an associated cluster algebra $\mathscr{A}$ and upper cluster algebra $\mathscr{U}$. By construction, $\mathscr{U}$ is isomorphic to the algebra of regular functions on the cluster $K_2$-variety $\mathcal{A}$. The following fact will be used in our proof of Theorem~\ref{thm:introcanonicalbasis}.

\begin{theorem}[\cite{Muller}, Corollary~11.5]
\label{thm:UequalsA}
If $|\mathbb{M}|>1$ then $\mathscr{U}=\mathscr{A}$.
\end{theorem}

A fundamental problem in the theory of cluster algebras is to express an arbitrary cluster variable $A_{l;t}$ in terms of the variables from some chosen initial seed. In the examples that we consider, this problem is solved by the following theorem, which is obtained from Corollary~6.3 of~\cite{FZIV} by taking the semifield to be trivial.

\begin{theorem}
\label{thm:expressclustervariable}
Let $\mathbf{s}_{t_0}=(Q,(A_1,\dots,A_m))$ be the initial seed and assume $Q$ has no frozen vertices. Then there exists a polynomial $F_l=F_{l;t}^{\varepsilon;t_0}\in\mathbb{Z}[u_1,\dots,u_m]$ and an integral vector $\mathbf{g}_l=\mathbf{g}_{l;t}^{\varepsilon;t_0}\in\mathbb{Z}^m$ such that the cluster variable $A_{l;t}$ can be written 
\[
A_{l;t}=F_l\left(\prod_iA_i^{\varepsilon_{1i}},\dots,\prod_iA_i^{\varepsilon_{mi}}\right)A_1^{g_1}\dots A_m^{g_m}
\]
where $\mathbf{g}_l=(g_1,\dots,g_m)$.
\end{theorem}

The polynomial $F_l=F_{l;t}^{\varepsilon;t_0}$ appearing in Theorem~\ref{thm:expressclustervariable} is called the \emph{$F$-polynomial} associated to the cluster variable $A_{l;t}$, and the integral vector $\mathbf{g}_l=\mathbf{g}_{l;t}^{\varepsilon;t_0}$ is called the \emph{$\mathbf{g}$-vector} associated to~$A_{l;t}$. They are important and well studied objects in the theory of cluster algebras.

\section{The canonical basis construction}
\label{sec:TheCanonicalBasisConstruction}

\subsection{Measured laminations}

We will be interested in a canonical basis for the algebra of regular functions on the cluster Poisson variety associated to a marked bordered surface. Geometrically, the elements of this canonical basis correspond to certain measured laminations on the surface~$\mathbb{S}$. We therefore begin by describing these objects.

In the following definition, when we talk about a \emph{curve} on~$\mathbb{S}$, we mean an embedding $C\rightarrow\mathbb{S}$ of a compact, connected, one-dimensional real manifold~$C$ with (possibly empty) boundary into~$\mathbb{S}$. We require that any endpoints of~$C$ map to points on the boundary of~$\mathbb{S}$ away from the marked points. When we talk about homotopies, we mean homotopies within the class of such curves. A curve is called \emph{special} if it cuts out a disk with a single marked point on its boundary. A curve is \emph{contractible} if it can be retracted to a point within this class of curves.

\begin{definition}[\cite{FG1,FG2}]
\label{def:integrallamination}
An \emph{integral lamination} on~$\Sigma=(\mathbb{S},\mathbb{M})$ is a collection of finitely many nonintersecting, noncontractible curves on~$\mathbb{S}$, either closed or ending on the boundary away from the marked points, with integral weights and subject to the following conditions:
\begin{enumerate}
\item The weight of a curve is nonnegative unless the curve is special.
\item If $e$ is a boundary segment, then the total weight of the curves ending on~$e$ vanishes.
\end{enumerate}
Moreover, we impose the following equivalence relations:
\begin{enumerate}
\item A lamination is equivalent to any lamination obtained by modifying its curves by homotopy.
\item A lamination containing a curve of weight zero is equivalent to the lamination with this curve removed.
\item A lamination containing homotopic curves of weights $a$ and $b$ is equivalent to the lamination with one curve removed and the weight $a+b$ on the other.
\end{enumerate}
\end{definition}

Suppose we are given an integral lamination $\ell$ on~$\Sigma$. If $T$ is an ideal triangulation of~$\Sigma$, then we can deform the curves of $\ell$ to get an equivalent lamination where each curve intersects the arcs of the triangulation transversely in the minimal number of points. Then for each arc $i$ of~$T$, we can define 
\[
a_i(\ell)=\frac{1}{2}\mu_i(\ell)
\]
where $\mu_i(\ell)$ is the total weight of the curves of the lamination that intersect~$i$.

\begin{definition}
Given the marked bordered surface $\Sigma$, we will write $\mathcal{A}(\mathbb{Z}^t)$ for the set of equivalence classes of laminations $\ell$ on $\Sigma$ such that the $a_i(\ell)$ are integers for any choice of triangulation.
\end{definition}

The peculiar notation used here is a reflection of the fact that $\mathcal{A}(\mathbb{Z}^t)$ can be understood as the set of points of the cluster $K_2$-variety $\mathcal{A}$ valued in the semifield $\mathbb{Z}^t$ of integers with the usual operations of addition and multiplication replaced by their ``tropical'' analogs $\max$ and $+$, respectively. Further details can be found in~\cite{FG1,FG2}.

\begin{proposition}[\cite{FG2}]
\label{prop:acoordinates}
For any choice of triangulation $T$, the functions $a_i$ provide a bijection 
\[
\mathcal{A}(\mathbb{Z}^t)\stackrel{\sim}{\longrightarrow}\mathbb{Z}^n
\]
where $n$ is the number of arcs in $T$.
\end{proposition}

\subsection{Loop functions}
\label{sec:loopfunctions}

Suppose $(\mathcal{L},s)$ is a general framed $PGL_2(\mathbb{C})$-local system on~$\Sigma$. Given an ideal triangulation~$T$ of $\Sigma$, let us label the arcs of~$T$ by the numbers $1,\dots,n$ and write $X_j\in\mathbb{C}^*$ for the coordinate of $(\mathcal{L},s)$ corresponding to the $j$th arc of~$T$. Let us choose a square root $X_j^{1/2}$ of~$X_j$ for each~$j$.

Suppose that $\ell$ is an integral lamination on $\Sigma$ consisting of a single closed curve $c$. By deforming the curve $c$ if necessary, we may assume that $c$ intersects each arc of the triangulation~$T$ transversely in the minimal number of points. Let $j_1,\dots,j_s$ be the arcs of $T$ that $c$ intersects, ordered according to some choice of orientation for the loop (so an arc may appear more than once on this list). After crossing the arc $j_k$, the curve $c$ enters a triangle~$t$ of~$T$ before leaving by crossing the next arc. If the curve $c$ turns to the left before leaving $t$, as depicted on the left hand side of Figure~\ref{fig:leftrightturn}, then we define 
\[
M_k = 
\left(
\begin{array}{cc}
X_{j_k}^{1/2} & X_{j_k}^{1/2} \\
0 & X_{j_k}^{-1/2} 
\end{array}
\right).
\]
On the other hand, if the curve $c$ turns to the right before leaving $t$, as depicted on the right hand side of Figure~\ref{fig:leftrightturn}, then we define 
\[
M_k = 
\left(
\begin{array}{cc}
X_{j_k}^{1/2} & 0 \\
X_{j_k}^{-1/2}  & X_{j_k}^{-1/2} 
\end{array}
\right).
\]
It is well known that the product 
\[
\rho(c)=M_1\dots M_s
\]
is a matrix representing the monodromy of~$\mathcal{L}$ around~$c$.

\begin{figure}[ht]
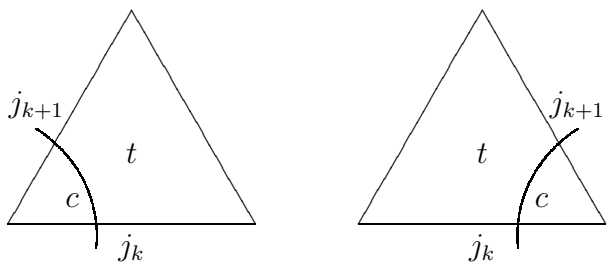
 \begin{center}
\[
\xy /l1.5pc/:
{\xypolygon3"A"{~:{(-3,0):}}},
{\xypolygon3"B"{~:{(-3.5,0):}~>{}}},
(3,-0.5)*{}="s";
(1.75,2)*{}="t";
(3,-1)*{j_{k+1}};
(1,2)*{j_k};
"s";"t" **\crv{(1.5,0.5) & (1.75,2)};
(2.25,1)*{c};
(1,0)*{t};
\endxy
\qquad
\xy /l1.5pc/:
{\xypolygon3"A"{~:{(-3,0):}}},
{\xypolygon3"B"{~:{(-3.5,0):}~>{}}},
(-1,-0.5)*{}="s";
(0.25,2)*{}="t";
(-1,-1)*{j_{k+1}};
(1,2)*{j_k};
"s";"t" **\crv{(0.5,0.5) & (0.25,2)};
(-0.25,1)*{c};
(1,0)*{t};
\endxy
\]
\caption{A left and right turn over the triangle $t$.\label{fig:leftrightturn}}
\end{center} \end{figure}

\begin{definition}
Let $\ell$ be an integral lamination on $\Sigma$ consisting of a single closed curve~$c$ of weight~$k$. Then we define 
\[
\mathbb{I}(\ell)\coloneqq\Tr(\rho(c)^k).
\]
\end{definition}

Thus $\mathbb{I}(\ell)$ is a Laurent polynomial in the variables $X_j^{1/2}$. In fact, since the monodromy $\rho(c)$ is a product of the matrices $M_k$ which factor as 
\[
\left(
\begin{array}{cc}
X_{j_k} & X_{j_k} \\
0  & 1 
\end{array}
\right)\cdot X_{j_k}^{-1/2}
\quad
\text{or}
\quad
\left(
\begin{array}{cc}
X_{j_k} & 0 \\
1  & 1 
\end{array}
\right)\cdot X_{j_k}^{-1/2},
\]
we see that the monodromy factors as $\rho(c)=M\cdot X_{j_1}^{-1/2}\dots X_{j_s}^{-1/2}$ for some matrix $M$ with polynomial entries. Therefore, if we write $F_\ell(X_1,\dots,X_n)$ for the trace of $M^k$, we obtain the following result.

\begin{lemma}
\label{lem:nonperipheralFpolynomial}
Let $\ell$ be an integral lamination on $\Sigma$ consisting of a single closed curve~$c$ of weight~$k$. Then
\[
\mathbb{I}(\ell)=F_\ell(X_1,\dots,X_n)\cdot X_1^{h_{\ell,1}}\dots X_n^{h_{\ell,n}}
\]
where $F_\ell(X_1,\dots,X_n)$ is a polynomial in the coordinates and $h_{\ell,j}$ is $-k/2$ times the geometric intersection number of $c$ and the arc~$j$ (the minimal number of intersections between an element in the homotopy class of $c$ and an element of the homotopy class of $j$ relative to its endpoints).
\end{lemma}

\subsection{Arc functions}
\label{sec:arcfunctions}

Now suppose that $\ell$ is an integral lamination on $\Sigma$ consisting of curves $c_0,c_1,\dots,c_N$ connecting points on the boundary of~$\mathbb{S}$. We assume that $c_0$ is a curve of weight $+1$ connecting the boundary segment~$b_0$ to the boundary segment~$b_1$, that $c_1$ is a curve of weight $-1$ connecting $b_1$ to the boundary segment~$b_2$, and that this pattern continues with the wights of the curves alternating between~$+1$ and $-1$ until we come to~$c_N$, which is a curve of weight~$-1$ connecting $b_N$ to $b_{N+1}=b_0$.

We will write $\gamma_k$ for the arc or boundary segment obtained from $c_k$ by dragging each of its endpoints along the boundary in the counterclockwise direction until it hits a marked point. If we are given a framed $PGL_2(\mathbb{C})$-local system $(\mathcal{L},s)$ on $\Sigma$, then we can find a decorated twisted $SL_2(\mathbb{C})$-local system $(\widehat{\mathcal{L}},\widehat{s})$ which is sent to $(\mathcal{L},s)$ by the map~\eqref{eqn:pmap}. Then by the construction of Definition~\ref{def:Acoordinate}, we associate a number $A_{\gamma_k}$ to each~$\gamma_k$.

\begin{definition}
Let $\ell$ be an integral lamination consisting of open curves $c_0,\dots,c_N$ as above. Then we define 
\[
\mathbb{I}(\ell)\coloneqq\prod_kA_{\gamma_k}^{w_k}
\]
where $w_k$ is the weight of the curve $c_k$.
\end{definition}

Let $T$ be an ideal triangulation of~$\Sigma$. By adapting the argument from Section~6.2.1 of~\cite{Allegretti18}, we will now show that the function $\mathbb{I}(\ell)$ can be written as a Laurent polynomial in the square roots $X_j^{1/2}$ associated to arcs of~$T$. In particular, it will follow that $\mathbb{I}(\ell)$ is independent of the decorated twisted local system that we chose in order to define it.

In order to apply the results from the theory of cluster algebras that we reviewed above, we will in fact consider the enlarged marked bordered surface $\overline{\Sigma}$ obtained from $\Sigma$ by gluing a triangle to each boundary segment of $\Sigma$ along one of its edges. This $\overline{\Sigma}$ is a marked bordered surface with twice as many boundary segments and marked points as~$\Sigma$. The ideal triangulation $T$ determines an ideal triangulation~$\overline{T}$ of~$\overline{\Sigma}$ in an obvious way, and the associated quiver $Q(\overline{T})$ can be regarded as an ice quiver with no frozen vertices and used to define a cluster algebra associated to the surface. In particular, there is an $F$-polynomial and a $\mathbf{g}$-vector associated to each $\gamma_k$, which is necessarily an arc in the enlarged surface~$\overline{\Sigma}$.

\begin{lemma}
\label{lem:expandAcoordinate}
For every $\gamma_k$, we have 
\[
A_{\gamma_k}=F_{\gamma_k}(X_1,\dots,X_n)\prod_i A_i^{g_{\gamma_k,i}}
\]
where $F_{\gamma_k}$ is the $F$-polynomial associated to $\gamma_k$ and $\mathbf{g}_{\gamma_k}=(g_{\gamma_k,i})$ is the $\mathbf{g}$-vector associated to~$\gamma_k$.
\end{lemma}

\begin{proof}
Let us label the edges of $T$ by the numbers $1,\dots,m$. Then we can view the coordinate functions $A_1,\dots,A_m$ as cluster variables in an initial seed for the cluster algebra associated to~$\overline{\Sigma}$. The function $A_{\gamma_k}$ is also a cluster variable in this algebra, and so by Theorem~\ref{thm:expressclustervariable}, we have 
\[
A_{\gamma_k}=F_{\gamma_k}\left(\prod_iA_i^{\varepsilon_{1i}},\dots,\prod_iA_i^{\varepsilon_{mi}}\right)A_1^{g_{\gamma_k,1}}\dots A_m^{g_{\gamma_k,m}}.
\]
If $\gamma_k$ is a boundary segment of~$\Sigma$, then we have $F_{\gamma_k}=1$ and we are done. Otherwise, it follows from the matrix formula of~\cite{MW} that $F_{\gamma_k}$ is a polynomial only in the variables associated to edges that $\gamma_k$ crosses, which are necessarily arcs of~$T$. For any arc $j$ of~$T$, the product $\prod_iA_i^{\varepsilon_{ji}}$ equals the $X$-coordinate associated to this arc.
\end{proof}

Let $c_k'$ be the curve on~$\overline{\Sigma}$ obtained by modifying $\gamma_k$ in a neighborhood of each of its endpoints, dragging the endpoint away from the marked point in the counterclockwise direction onto the adjacent boundary segment of~$\overline{\Sigma}$.

\begin{lemma}
\label{lem:productgvectors}
Let $\mathbf{s}=(s_i)$ be the vector defined by $\mathbf{s}=\sum_{k\text{ even}}\mathbf{g}_{\gamma_k}-\sum_{k\text{ odd}}\mathbf{g}_{\gamma_k}$. Then there exists a half integral vector $\mathbf{h}=(h_{\ell,j})$, indexed by the arcs $j$ of~$T$, such that 
\[
\prod_iA_i^{s_i}=\prod_jX_j^{h_{\ell,j}}.
\]
Namely, $h_{\ell,j}$ is given by $-\frac{1}{2}\sum_kw_k\mu_{k,j}$ where $\mu_{k,j}$ is the geometric intersection number of $c_k'$ and the arc~$j$ (the minimal intersection number relative to endpoints).
\end{lemma}

\begin{proof}
Let $c$ be a curve on a marked bordered surface and let $T$ be an ideal triangulation of this surface. We assume that $c$ intersects the arcs of $T$ transversely and the number of intersections with a given edge is equal to the geometric intersection number. If $j$ is any arc of~$T$, then there is a quadrilateral $q$ formed by the two triangles that share this edge~$j$. The \emph{shear parameter} at $j$ is defined as a sum of contributions from all intersections of $c$ with the arc~$j$. Specifically, such an intersection contributes $+1$ (respectively, $-1$) if the curve connects opposite sides of~$q$ in the manner illustrated on the left (respectively, right) hand side of Figure~\ref{fig:shear}.
\begin{figure}[ht]
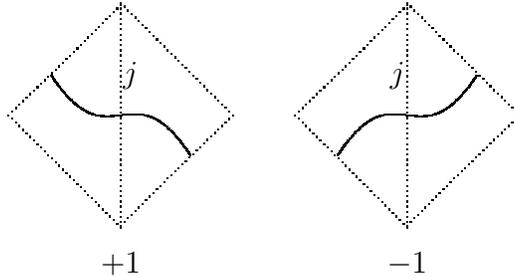
 \begin{center}
\[
\xy /l1.25pc/:
{\xypolygon4"A"{~:{(2,2):}~>{.}}};
{"A1"\PATH~={**@{.}}'"A3"};
(-0.75,1)*{}="A0";
(2.75,-1)*{}="B0";
(1,3.75)*{+1};
(0.75,-1)*{j};
"A0";"B0" **\crv{(0.75,-1.25) & (1.25,1.25)};
\endxy
\qquad
\xy /l1.25pc/:
{\xypolygon4"A"{~:{(2,2):}~>{.}}};
{"A1"\PATH~={**@{.}}'"A3"};
(-0.75,-1)*{}="A0";
(2.75,1)*{}="B0";
(1,3.75)*{-1};
(1.25,-1)*{j};
"A0";"B0" **\crv{(0.75,1.25) & (1.25,-1.25)};
\endxy
\]
\caption{Two types of contributions to the shear parameter.\label{fig:shear}}
\end{center} \end{figure}
In particular, we can talk about the shear parameter associated to the curve $c_k'$ and an arc $j$ of~$\overline{T}$. According to~\cite{Reading}, Proposition~5.2, the component $g_{\gamma_k,j}$ of the $\mathbf{g}$-vector $\mathbf{g}_{\gamma_k}$ equals this shear parameter. Using this result and the definition of the $X_j$ as monomials in the $A_i$, one can check that 
\[
\prod_jX_j^{\eta_j}=\prod_iA_i^{-2s_i}
\]
where $\eta_j=\sum_kw_k\mu_{k,j}$.
\end{proof}

\begin{lemma}
\label{lem:arcFpolynomial}
Let $\ell$ be an integral lamination consisting of open curves $c_0,\dots,c_N$ as above. Then 
\[
\mathbb{I}(\ell)=F_\ell(X_1,\dots,X_n)\cdot X_1^{h_{\ell,1}}\dots X_n^{h_{\ell,n}}
\]
where 
\[
F_\ell(X_1,\dots,X_n)=\prod_{k=0}^NF_{\gamma_k}(X_1,\dots,X_n)
\]
and $h_{\ell,j}$ is defined as in Lemma~\ref{lem:productgvectors}.
\end{lemma}

\begin{proof}
By Lemma~\ref{lem:expandAcoordinate} and the definition of $\mathbb{I}(\ell)$, we have 
\[
\mathbb{I}(\ell)=\prod_{k=0}^N\left(F_{\gamma_k}(X_1,\dots,X_n)\cdot\prod_i A_i^{g_{\gamma_k,i}}\right)^{w_k}.
\]
By assumption, we have either $w_k=+1$ or $w_k=-1$. If $w_k=-1$ then $\gamma_k$ is a boundary segment on~$\Sigma$, and the associated polynomial is $F_{\gamma_k}=1$. Therefore 
\[
\mathbb{I}(\ell)=\left(\prod_{k=0}^NF_{\gamma_k}(X_1,\dots,X_n)\right)\cdot \prod_iA_i^{s_i}
\]
where $s_i=\sum_{k\text{ even}}g_{\gamma_k,i}-\sum_{k\text{ odd}}g_{\gamma_k,i}$. The result now follows from Lemma~\ref{lem:productgvectors}.
\end{proof}

\subsection{The general case}

We can now give the general definition of the canonical basis for the algebra of regular functions on the cluster Poisson variety. In the following discussion, two integral laminations $\ell_1$ and $\ell_2$ will be called \emph{compatible} if no curve from $\ell_1$ intersects or is homotopic to a curve from~$\ell_2$. If $\ell_1$ and~$\ell_2$ are compatible laminations, we write $\ell_1+\ell_2$ for the lamination defined by taking the union of the curves from $\ell_1$ and $\ell_2$.

\begin{definition}
If $\ell_1$ and $\ell_2$ are compatible laminations, then we define 
\[
\mathbb{I}(\ell_1+\ell_2)\coloneqq\mathbb{I}(\ell_1)\mathbb{I}(\ell_2).
\]
\end{definition}

If $\ell$ is any integral lamination on~$\Sigma$, then $\ell$ can be represented as a union of curves of weight~$\pm1$. Recalling condition~2 in Definition~\ref{def:integrallamination}, we can inductively remove laminations of the types considered in Subsections~\ref{sec:loopfunctions} and~\ref{sec:arcfunctions}, and we see that $\ell$ can be written as a sum 
\begin{align}
\label{eqn:decomposelamination}
\ell=\sum_i\ell_i
\end{align}
where the $\ell_i$ are pairwise compatible laminations of these two types. Thus $\mathbb{I}(\ell)$ is defined for any lamination~$\ell$. The next result follows immediately from the definition and Lemmas~\ref{lem:nonperipheralFpolynomial} and~\ref{lem:arcFpolynomial}.

\begin{proposition}
\label{prop:generalFpolynomial}
If $\ell$ is any integral lamination on $\Sigma$ written in the form~\eqref{eqn:decomposelamination}, then 
\[
\mathbb{I}(\ell)=\prod_iF_{\ell_i}(X_1,\dots,X_n)\cdot X_1^{h_{\ell,1}}\dots X_n^{h_{\ell,n}}
\]
where $h_{\ell,k}=\sum_ih_{\ell_i,k}$.
\end{proposition}

Note in particular that if $\ell\in\mathcal{A}(\mathbb{Z}^t)$, then the total weight of the curves of $\ell$ that intersect a given edge is always even and therefore, for such laminations $\ell$, $\mathbb{I}(\ell)$ is a Laurent polynomial in the coordinates $X_i$ and is independent of the choice of square roots $X_i^{1/2}$. Thus the above construction gives a canonical map 
\[
\mathbb{I}:\mathcal{A}(\mathbb{Z}^t)\rightarrow\mathcal{O}(\mathcal{X}).
\]
We claim that the image of this map is a canonical $\mathbb{Q}$-vector space basis for the right hand side. We will now prove this in several steps.

To prove that the functions $\mathbb{I}(\ell)$ are linearly independent, we use the following well known fact.

\begin{lemma}[\cite{FG2}, Section~7.2]
\label{lem:highestterm}
Let $\ell$ be any integral lamination. Then for any ideal triangulation, $\mathbb{I}(\ell)$ can be written as a Laurent polynomial in the $X_i^{1/2}$ with positive integer coefficients. The highest order term of this Laurent polynomial is $X_1^{a_1}\dots X_n^{a_n}$ where $a_1,\dots,a_n$ are the coordinates of $\ell$.
\end{lemma}

\begin{proposition}
\label{prop:independent}
The functions $\mathbb{I}(\ell)$ for $\ell\in\mathcal{A}(\mathbb{Z}^t)$ are linearly independent.
\end{proposition}

\begin{proof}
Suppose $\alpha_1,\dots,\alpha_s$ are rational numbers and $\ell_1,\dots,\ell_n$ are distinct laminations such that 
\[
\sum_{k=1}^s\alpha_k\mathbb{I}(\ell_k)=0.
\]
We can impose a lexicographic total ordering on the set of all monic Laurent monomials in the variables~$X_j$. Then one of the canonical functions, say $\mathbb{I}(\ell_1)$, will have the maximal leading term with respect to this total ordering. If there is another lamination $\ell_k$ such that $\mathbb{I}(\ell_k)$ has the same leading term, then by Lemma~\ref{lem:highestterm}, the lamination $\ell_k$ has the same coordinates as~$\ell_1$. By Proposition~\ref{prop:acoordinates}, we must have $\ell_1=\ell_k$, contradicting the assumption that these laminations are distinct. It follows that the leading term of $\mathbb{I}(\ell_1)$ cannot cancel with any other term in the above sum, and we must have $\alpha_1=0$. Thus 
\[
\sum_{k=2}^s\alpha_k\mathbb{I}(\ell_k)=0
\]
and we can apply the same argument to this new sum. Continuing in this way, we see that $\alpha_k=0$ for $k=1,\dots,s$.
\end{proof}

To prove that the functions $\mathbb{I}(\ell)$ span the algebra of all regular functions, we consider a modification of the notion of an integral lamination in which the constituent curves are allowed to have intersections.

\begin{definition}
A \emph{quasi-lamination} on $\Sigma=(\mathbb{S},\mathbb{M})$ is a collection of finitely many arcs, boundary segments, and closed loops on~$\mathbb{S}$ with weights $\pm1$ such that all intersections are transverse and only boundary segments can have weight~$-1$.
\end{definition}

If $\ell$ is a quasi-lamination consisting of a single closed loop, then we can lift $\ell$ to a loop in the punctured tangent bundle $T'\mathbb{S}$. The monodromy of a decorated twisted $SL_2$-local system around this loop can be written as a Laurent polynomial in the coordinates $A_i$ associated to a choice of ideal triangulation~$T$, and this Laurent polynomial has all positive or all negative coefficients (see~\cite{FG1}, Theorem~12.2). Thus, after multiplying by $\pm1$, we get a Laurent polynomial with all positive coefficients which we call $\mathbb{I}'(\ell)$. On the other hand, if $\ell$ is a quasi-lamination consisting of a single arc or boundary segment $\gamma$ of weight~$w$, then we define $\mathbb{I}'(\ell)$ to be the function~$A_\gamma^w$, written in terms of the coordinates associated with the triangulation~$T$. Finally, if $\ell$ is a quasi-lamination which is obtained as a union of two quasi-laminations $\ell_1$ and $\ell_2$, then we define 
\[
\mathbb{I}'(\ell)=\mathbb{I}'(\ell_1)\mathbb{I}'(\ell_2).
\]
In this way, we associate to every quasi-lamination $\ell$ a Laurent polynomial $\mathbb{I}'(\ell)$ in the coordinates~$A_i$. The following fact was used in~\cite{FG1} in the proof of Theorem~12.2 (see also~\cite{MW}, Section~6).

\begin{lemma}
\label{lem:skein}
Let $\ell$ be a quasi-lamination on $\Sigma$, and let $p$ be a point of $\mathbb{S}$ where two curves of $\ell$ intersect so that locally around $p$ the lamination looks like the left hand side of Figure~\ref{fig:skein}. Let $\ell_1$ and $\ell_2$ be the quasi-laminations obtained by modifying $\ell$ in a neighborhood of $p$ as illustrated in the middle and right hand side of Figure~\ref{fig:skein}. Then we have 
\[
\mathbb{I}'(\ell)\pm\mathbb{I}'(\ell_1)\pm\mathbb{I}'(\ell_2)=0
\]
for some choice of signs depending on the quasi-laminations.
\end{lemma}

\begin{figure}[ht]
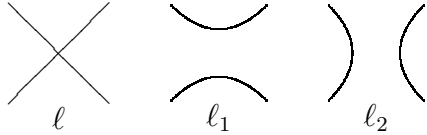
 \begin{center}
\[
\stackrel{
\xy /l1.5pc/:
{\xypolygon4"C"{~:{(1.5,0):}~>{}}};
{"C1"\PATH~={**@{-}}'"C3"};
{"C2"\PATH~={**@{-}}'"C4"};
\endxy
}
{\ell}
\qquad
\stackrel{
\xygraph{
    !{0;/r3.0pc/:}
    [u(0.5)]!{\xunoverh}
}}
{\ell_1}
\qquad
\stackrel{
\xygraph{
    !{0;/r3.0pc/:}
    [u(0.5)]!{\xunoverv}
}}
{\ell_2}
\]
\caption{Resolution of a crossing.\label{fig:skein}}
\end{center} \end{figure}

This lemma gives a \emph{skein relation} that we can use to express the Laurent polynomial associated with a quasi-lamination $\ell$ in terms of the Laurent polynomials associated with simpler quasi-laminations. Another important ingredient in our proof is the following fact.

\begin{lemma}
\label{lem:Chebyshev}
Let $\ell$ be a lamination consisting of a single closed loop of weight~1, and let $k\ell$ be the lamination defined by the same curve with weight~$k$. Then for $K>0$ we can write 
\[
\mathbb{I}(\ell)^K=\sum_{k=0}^Kc_k\mathbb{I}(k\ell)
\]
for $c_k\in\mathbb{Z}_{\geq0}$ such that $c_k=0$ if $k$ has parity different from~$K$.
\end{lemma}

\begin{proof}
Recall that the expression $\mathbb{I}(k\ell)$ is defined as $\Tr(\rho(\ell)^k)$ where $\rho(\ell)$ is a matrix in $SL_2(\mathbb{C})$ representing the monodromy around~$\ell$. It is well known (see~\cite{AllegrettiKim}, Proposition~3.6) that this trace is equal to $P_k(\Tr(\rho(\ell)))$ where $P_k$ is the \emph{Chebyshev polynomial} defined recursively by 
\[
P_0(t)=2, \quad P_1(t)=t, \quad P_{N+1}(t)=tP_N(t)-P_{N-1}(t).
\]
On the other hand, it was shown in~\cite{MSW2}, Proposition~2.35 (see also~\cite{AllegrettiKim}, Lemma~3.19) that there is an identity 
\[
t^K=\sum_{k=0}^Kc_kP_k(t)
\]
where $c_k\in\mathbb{Z}_{\geq0}$, and $c_k=0$ if $k$ has parity different from~$K$. The lemma follows.
\end{proof}

We now have all of the ingredients we need to prove that the functions $\mathbb{I}(\ell)$ span $\mathcal{O}(\mathcal{X})$.

\begin{proposition}
\label{prop:span}
If $|\mathbb{M}|>1$ then the functions $\mathbb{I}(\ell)$ span $\mathcal{O}(\mathcal{X})$.
\end{proposition}

\begin{proof}
Let $f\in\mathcal{O}(\mathcal{X})$ be any regular function on the cluster Poisson variety. Then for any ideal triangulation~$T$ of~$\Sigma$, we can express $f$ as a Laurent polynomial in the associated coordinates~$X_i$. Substituting $X_i=\prod_jA_j^{\varepsilon_{ij}}$, we get a Laurent polynomial 
\begin{align}
\label{eqn:expressioninU}
f=\sum_{\mathbf{d}=(d_\gamma)\in\mathbb{Z}^m}\lambda_{\mathbf{d}}\cdot\prod_{\gamma}A_\gamma^{d_\gamma}
\end{align}
where $\lambda_{\mathbf{d}}\in\mathbb{Q}$ and $\gamma$ runs over all edges of~$T$. Consider the term in this expression corresponding to a vector $\mathbf{d}=(d_\gamma)$. The sum $z_{\mathbf{d}}=\sum d_\gamma\gamma$ is a chain of unoriented singular 1-simplices defining a class $[z_{\mathbf{d}}]\in H_1(\mathbb{S},\mathbb{Z}/2\mathbb{Z})$, and since this term arises from a product of the cross ratios $X_i$, we in fact have $[z_{\mathbf{d}}]=0$.

Now the expression~\eqref{eqn:expressioninU} is an element of the upper cluster algebra $\mathscr{U}$. Since we assume $|\mathbb{M}|>1$, Theorem~\ref{thm:UequalsA} says that this upper cluster algebra $\mathscr{U}$ equals the cluster algebra~$\mathscr{A}$, and hence we can write 
\begin{align}
\label{eqn:expressioninA}
f=\sum_{\mathbf{e}}\zeta_\mathbf{e}\cdot\prod_\alpha A_\alpha^{e_\alpha}\prod_\beta A_\beta^{e_\beta}.
\end{align}
Here $\alpha$ runs over all arcs on $\Sigma$ (not necessarily belonging to the ideal triangulation $T$) and we have $e_\alpha\in\mathbb{Z}_{\geq0}$, while $\beta$ runs over all boundary segments and we have $e_\beta\in\mathbb{Z}$. The coefficient $\zeta_{\mathbf{e}}$ is a rational number. Let us consider the sum $w_{\mathbf{e}}=\sum e_\alpha\alpha+\sum e_\beta\beta$ which defines a class $[w_{\mathbf{e}}]\in H_1(\mathbb{S},\mathbb{Z}/2\mathbb{Z})$. The expression~\eqref{eqn:expressioninA} is obtained from~\eqref{eqn:expressioninU} by applying relations in the algebra~$\mathscr{U}$. These relations have the form 
\[
A_{\gamma_0} A_{\gamma_0'} = A_{\gamma_1}A_{\gamma_3}+A_{\gamma_2}A_{\gamma_4}
\]
where $\gamma_0,\gamma_0',\gamma_1,\dots,\gamma_4$ form a quadrilateral as illustrated in Figure~\ref{fig:quadrilateral}.
\begin{figure}[ht]
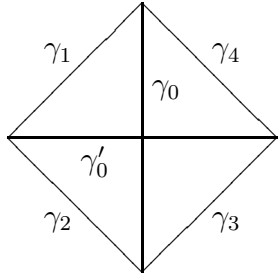
 \begin{center}
\[
\xy /l1.5pc/:
{\xypolygon4"A"{~:{(2,2):}}};
{\xypolygon4"B"{~:{(2,2):}~>{}}};
{\xypolygon4"C"{~:{(2.5,0):}~>{}}};
{"A1"\PATH~={**@{-}}'"A3"};
{"A2"\PATH~={**@{-}}'"A4"};
"C1"*{\gamma_2},
"C2"*{\gamma_3},
"C3"*{\gamma_4},
"C4"*{\gamma_1},
(0.5,-1)*{\gamma_0},
(2,0.5)*{\gamma_0'},
\endxy
\]
\caption{Labeling of a quadrilateral.\label{fig:quadrilateral}}
\end{center} \end{figure}
It is easy to see that the singular 1-chain $\gamma_0+\gamma_0'$ is homologous to both $\gamma_1+\gamma_3$ and $\gamma_2+\gamma_4$. We have already seen that $[z_{\mathbf{d}}]=0$ in $H_1(\mathbb{S},\mathbb{Z}/2\mathbb{Z})$, and therefore we have $[w_{\mathbf{e}}]=0$ as well.

Consider the term of~\eqref{eqn:expressioninA} corresponding to the vector $\mathbf{e}$. We associate a quasi-lamination~$\ell_{\mathbf{e}}$ to this term as follows. This quasi-lamination consists of $e_\alpha$ copies of the arc $\alpha$ and $|e_\beta|$ copies of the boundary segment $\beta$. All arcs have weight $+1$ while a boundary segment $\beta$ has weight~$\pm1$ according to the sign of~$e_\beta$. By what we have said so far, this quasi-lamination defines a singular 1-cycle representing the zero class in $H_1(\mathbb{S},\mathbb{Z}/2\mathbb{Z})$. The arcs of $\ell_{\mathbf{e}}$ may intersect, but we can apply Lemma~\ref{lem:skein} to write 
\[
\prod_\alpha A_\alpha^{e_\alpha}\prod_\beta A_\beta^{e_\beta}=\mathbb{I}'(\ell_{\mathbf{e}})=\sum_{i=1}^N\kappa_i\mathbb{I}'(\ell_i)
\]
where the $\ell_i$ are quasi-laminations without intersections and $\kappa_i\in\mathbb{Z}$. Note that the three quasi-laminations appearing in Lemma~\ref{lem:skein} are homologous, and therefore each $\ell_i$ represents the zero class in homology. Let $\ell_{i,1},\dots,\ell_{i,r}$ be the closed loops appearing in the quasi-lamination $\ell_i$ with multiplicities $K_{i,1},\dots,K_{i,r}$, respectively, and let $\ell_{i,0}$ be the quasi-lamination formed by all remaining arcs and boundary segments in~$\ell_i$. Then we can write 
\[
\mathbb{I}'(\ell_i)=\pm\mathbb{I}(\ell_{i,0})\mathbb{I}(\ell_{i,1})^{K_{i,1}}\dots\mathbb{I}(\ell_{i,r})^{K_{i,r}}.
\]
(Here, by abuse of notation, we use the same symbol for a lamination and for the quasi-lamination obtained by dragging the endpoints of all arcs and boundary segments along~$\partial\mathbb{S}$ in the counterclockwise direction until they hit the marked points.) By Lemma~\ref{lem:Chebyshev}, we can expand each power $\mathbb{I}(\ell_{i,s})^{K_{i,s}}$ as a linear combination of functions $\mathbb{I}(k_{i,s}\ell_{i,s})$ where $k_{i,s}$ has the same parity as~$K_{i,s}$. Thus we expand $\mathbb{I}'(\ell_i)$ as a linear combination of terms of the form 
\[
\mathbb{I}(\ell_{i,0})\mathbb{I}(k_{i,1}\ell_{i,1})\dots\mathbb{I}(k_{i,r}\ell_{i,r})=\mathbb{I}(\ell_{i,0}+k_{i,1}\ell_{i,1}+\dots+k_{i,r}\ell_{i,r}).
\]
The lamination appearing on the right hand side of this expression determines the same class in homology as $\ell_i$, namely the zero class. It is known (see~\cite{thesis}, Theorem~5.2.1) that an integral lamination lies in~$\mathcal{A}(\mathbb{Z}^t)$ if and only if its class in $H_1(\mathbb{S},\mathbb{Z}/2\mathbb{Z})$ is zero. Thus we have shown that each term in~\eqref{eqn:expressioninA} is a linear combination of the functions $\mathbb{I}(\ell)$ for $\ell\in\mathcal{A}(\mathbb{Z}^t)$. This completes the proof.
\end{proof}

Combining Propositions~\ref{prop:independent} and~\ref{prop:span}, we obtain the following.

\begin{theorem}
If $|\mathbb{M}|>1$ then the functions $\mathbb{I}(\ell)$ for $\ell\in\mathcal{A}(\mathbb{Z}^t)$ form a canonical $\mathbb{Q}$-vector space basis for the algebra $\mathcal{O}(\mathcal{X})$.
\end{theorem}

\section{Representations of quivers}
\label{sec:RepresentationsOfQuivers}

\subsection{The Jacobian algebra}

Here we collect some elementary facts about quivers with potential and their Jacobian algebras with an emphasis on the quivers with potential associated to triangulated surfaces. Further details on this material can be found in~\cite{DWZ1} and~\cite{L1}.

Recall that a quiver can be viewed as a quadruple $Q=(Q_0,Q_1,s,t)$ where $Q_0$ is the set of vertices, $Q_1$ is the set of arrows, and the maps $s:Q_1\rightarrow Q_0$ and $t:Q_1\rightarrow Q_0$ take an arrow to its source and target, respectively.

\begin{definition}
A \emph{path} in $Q$ is a sequence of arrows $a_0,\dots,a_k$ such that $s(a_i)=t(a_{i-1})$ for $i=1,\dots,k$. We will denote this path by $p=a_k\dots a_0$. The \emph{source} of $p$ is defined by $s(p)=s(a_0)$, and the \emph{target} is defined by $t(p)=t(a_k)$. Two paths $p$ and $q$ are said to be \emph{compatible} if $s(p)=t(q)$, and in this case the composition $pq$ is defined by juxtaposition. The \emph{path algebra} $\mathbb{C}Q$ is the algebra over $\mathbb{C}$ generated by paths in~$Q$, equipped with a multiplication where the product of paths $p$ and $q$ is the composition $pq$ if these paths are compatible and zero otherwise.
\end{definition}

Given a quiver $Q$, the set of arrows defines an ideal in the path algebra $\mathbb{C}Q$ and we write $\widehat{\mathbb{C}Q}$ for the completion of $\mathbb{C}Q$ with respect to this ideal. We regard $\widehat{\mathbb{C}Q}$ as a space of possibly infinite linear combinations of paths in~$Q$ with multiplication defined by concatenation.

\begin{definition}
A \emph{potential} for~$Q$ is an element of $\widehat{\mathbb{C}Q}$ all of whose terms are cyclic paths of positive length. Two potentials are said to be \emph{cyclically equivalent} if their difference lies in the closure of the vector subspace of $\widehat{\mathbb{C}Q}$ spanned by elements of the form $a_1\dots a_d-a_2\dots a_da_1$ where $a_1\dots a_d$ is a cyclic path. A \emph{quiver with potential} is a pair $(Q,W)$ where $Q$ is a quiver and $W$ is a potential for~$Q$ such that no two terms of $W$ are cyclically equivalent.
\end{definition}

In~\cite{L1}, Labardini-Fragoso associated a quiver with potential to any ideal triangulation~$T$ of a marked bordered surface $\Sigma$. We recall the definition of this quiver with potential in the case where $\Sigma$ is a marked bordered surface without punctures. Note that in this case if $t$ is a triangle of $T$ whose edges are all arcs, then there is an oriented 3-cycle $C(t)$ of~$Q(T)$ whose vertices are the edges of $t$. Thus we can define a natural potential $W(T)$ for~$Q(T)$ by 
\[
W(T)=\sum_{t}C(t)
\]
where the sum is over all triangles of $T$ whose edges are arcs. Labardini-Fragoso showed in~\cite{L1} that if two ideal triangulations are related by a flip, then the associated quivers with potential are related by a mutation in the sense of~\cite{DWZ1}.

\begin{definition}
If $a\in Q_1$ is any arrow and $p=a_1\dots a_d$ is a cyclic path in~$Q$, then the \emph{cyclic derivative} of~$p$ with respect to~$a$ is given by the expression 
\[
\partial_a(p)=\sum_{i:a_i=a}a_{i+1}\dots a_da_1\dots a_{i-1}.
\]
\end{definition}

This operation extends by linearity to give a map $\partial_a$ from the vector space of potentials to $\widehat{\mathbb{C}Q}$ such that any two cyclically equivalent potentials have the same image.

\begin{definition}
Let $(Q,W)$ be a quiver with potential. The \emph{Jacobian ideal} is defined as the closure of the two-sided ideal in $\widehat{\mathbb{C}Q}$ generated by the set $\{\partial_aW:a\in Q_1\}$. The \emph{Jacobian algebra} $J(Q,W)$ is defined as the quotient of $\widehat{\mathbb{C}Q}$ by the Jacobian ideal.
\end{definition}

\begin{definition}
An ideal $\mathfrak{a}\subseteq\mathbb{C}Q$ is said to be \emph{admissible} if we have $\mathfrak{b}^k\subseteq\mathfrak{a}\subseteq\mathfrak{b}^2$ for some $k\geq2$ where $\mathfrak{b}$ is the ideal generated by all arrows in $\mathbb{C}Q$.
\end{definition}

In the case of marked bordered surfaces without punctures, the results of Labardini-Fragoso give the following.

\begin{theorem}[\cite{L1}, Theorem~36]
\label{thm:Jfinitedimensional}
Let $(Q(T),W(T))$ be the quiver with potential associated to an ideal triangulation $T$ of a marked bordered surface without punctures. Then the Jacobian algebra $J(Q(T),W(T))$ is finite dimensional. Moreover, the ideal $\mathfrak{a}$ generated by the set $\{\partial_aW(T):a\in (Q(T))_1\}$ in the (uncompleted) path algebra $\mathbb{C}Q(T)$ is admissible, and $J(Q(T),W(T))\cong\mathbb{C}Q(T)/\mathfrak{a}$.
\end{theorem}

As a consequence of this result, we see that $J(Q(T),W(T))$ belongs to a class of algebras whose representation theory is particularly well studied.

\begin{definition}
A finite dimensional algebra $A$ is called a \emph{string algebra} if there exists a quiver $Q$ and an admissible ideal $\mathfrak{a}$ such that $A\cong\mathbb{C}Q/\mathfrak{a}$ and the following conditions hold:
\begin{enumerate}
\item If $i$ is a vertex of $Q$, then there are at most two arrows that start at~$i$ and at most two arrows that end at~$i$.
\item For each arrow $a$, there is at most one arrow $b$ and at most one arrow $c$ such that $ab\not\in\mathfrak{a}$ and $ca\not\in\mathfrak{a}$.
\end{enumerate}
\end{definition}

The following concepts play a distinguished role in the representation theory of a string algebra.

\begin{definition}
For any arrow $a$ of $Q$, we denote by $a^{-1}$ a formal inverse of~$a$ with $s(a^{-1})=t(a)$ and $t(a^{-1})=s(a)$. A word $w=a_Na_{N-1}\dots a_1$ in the arrows and their formal inverses is called a \emph{string} if we have $t(a_i)=s(a_{i+1})$ and $a_{i+1}\neq a_i^{-1}$ for all indices~$i$, and no subword or its inverse belongs to the ideal~$\mathfrak{a}$. A string $w$ is \emph{cyclic} if $s(a_1)=t(a_N)$. A \emph{band} is defined to be a cyclic string $w$ such that each power $w^k$ is a string but $w$ itself is not a proper power of any string.
\end{definition}

Below we will be interested in representations of a string algebra associated to strings and bands.

\begin{definition}
A \emph{representation} $M$ of a quiver $Q=(Q_0,Q_1,s,t)$ consists of a finite-dimensional $\mathbb{C}$-vector space $M(i)$ for each $i\in Q_0$ and a $\mathbb{C}$-linear map $M(a):M(s(a))\rightarrow M(t(a))$ for each $a\in Q_1$. The \emph{dimension vector} of $M$ is the tuple of nonnegative integers given by 
\[
\underline{\dim} M = (\dim_{\mathbb{C}} M(i))_{i\in Q_0}.
\]
\end{definition}

Representations of a quiver $Q$ can be viewed as modules over the path algebra $\mathbb{C}Q$. Indeed, given a representation $M$ of the quiver $Q$, we get a module over $\mathbb{C}Q$ whose underlying vector space is $M=\bigoplus_{i\in Q_0}M(i)$. One can show that this correspondence between representations of~$Q$ and modules over~$\mathbb{C}Q$ in fact provides an equivalence of categories.

By a \emph{morphism} $f:Q'\rightarrow Q$ of quivers $Q'=(Q_0',Q_1',s',t')$ and $Q=(Q_0,Q_1,s,t)$, we mean a pair of maps $f_0:Q_0'\rightarrow Q_0$ and $f_1:Q_1'\rightarrow Q_1$ such that $f_0\circ s'=s\circ f_1$ and $f_0\circ t'=t\circ f_1$. Given a morphism $f:Q'\rightarrow Q$ and a representation $M'$ of $Q'$, we get a corresponding representation $f_*(M')$ of $Q$ where, for any $i\in Q_0$ and any $a\in Q_1$, we define 
\[
(f_*(M'))(i)\coloneqq\bigoplus_{i'\in f_0^{-1}(i)}M'(i'), \qquad (f_*(M'))(a)\coloneqq\bigoplus_{a'\in f_1^{-1}(a)}M'(a').
\]
Note that if $A\cong\mathbb{C}Q/\mathfrak{a}$ is a string algebra and $w$ is a string for~$A$, then $w$ determines in a natural way a quiver $Q'$ of Dynkin type $A$ and a morphism of quivers $f:Q'\rightarrow Q$ satisfying the following properties:
\begin{enumerate}
\item If $a$,~$b\in Q_1'$ are arrows with $a\neq b$ and $s'(a)=s'(b)$, then $f_1(a)\neq f_1(b)$.
\item If $a$,~$b\in Q_1'$ are arrows with $a\neq b$ and $t'(a)=t'(b)$, then $f_1(a)\neq f_1(b)$.
\end{enumerate}
If $M'$ is the representation of this quiver $Q'$ defined by setting $M'(i)=\mathbb{C}$ for all vertices $i$ and $M'(a)=\mathrm{id}_{\mathbb{C}}$ for all arrows~$a$, then this construction gives a representation $f_*(M')$ of~$Q$ which we can regard as a module over~$A$.

\begin{definition}
The module defined by the above construction is called the \emph{string module} associated to~$w$.
\end{definition}

Similarly, if $w$ is a band, then there is a quiver $Q'$ of affine type $\widetilde{A}$ and a morphism $f:Q'\rightarrow Q$ satisfying the properties~1 and~2 above. For any choice of $n\in\mathbb{Z}_{>0}$ and $\phi\in\Aut(\mathbb{C}^n)$, we can define a representation $M'$ of this quiver $Q'$ by setting $M'(i)=\mathbb{C}^n$ for all vertices~$i$ and setting $M'(a)=\mathrm{id_{\mathbb{C}^n}}$ for all but one of the arrows $a\in Q_1'$ and $M'(a)=\phi$ for the remaining arrow $a\in Q_1'$. We then get the module $f_*(M')$ over~$A$.

\begin{definition}
An indecomposable module defined by the above construction is called a \emph{band module} associated to the band~$w$ and the data $n\in\mathbb{Z}_{>0}$ and~$\phi\in\Aut(\mathbb{C}^n)$.
\end{definition}

It is known~\cite{BrustleZhang} that every finite-dimensional indecomposable module over a string algebra is isomorphic to a string module or a band module. The string module associated to a string $w$ is isomorphic to a string module associated to the inverse of~$w$. Similarly, the band module associated to a band $w$ and $n\in\mathbb{Z}_{>0}$, $\phi\in\Aut(\mathbb{C}^n)$ is isomorphic to band modules obtained by reversing the orientation of $w$ and applying cyclic permutations. Apart from these, there are no isomorphisms between string and band modules. Note that in order for the band module defined above to be indecomposable, the automorphism $\phi$ must be expressible as an $n\times n$ Jordan block for some eigenvalue $\lambda\in\mathbb{C}^*$. The band module is then denoted $M(w,n,\lambda)$.

Finally, let us discuss the Grassmannian of subrepresentations.

\begin{definition}
Let $M$ be a representation of a quiver $Q$. By a \emph{subrepresentation}, we mean a submodule $N$ of~$M$, regarded as a module over the path algebra. Equivalently, $N$ consists of a subspace $N(i)\subseteq M(i)$ for each $i\in Q_0$ such that 
\[
M(a)(N(i))\subseteq N(j)
\]
for every arrow $a:i\rightarrow j$ in~$Q_1$. For any tuple $\mathbf{e}=(e_i)_{i\in Q_0}$ of nonnegative integers, we define the \emph{quiver Grassmannian} $\Gr_{\mathbf{e}}(M)$ to be the set of subrepresentations of $M$ with dimension vector~$\mathbf{e}$.
\end{definition}

The quiver Grassmannian is a closed subset of the product $\prod_{i=1}^n\Gr_{e_i}(M(i))$ of ordinary Grassmannians. Hence $\Gr_{\mathbf{e}}(M)$ is a complex projective variety. We will be particularly interested in the following subset, which was studied in~\cite{Dupont,CerulliEsposito,CerulliDupontEsposito} in the special case where $Q$ is affine.

\begin{definition}
If $M$ is an indecomposable module over a string algebra $\mathbb{C}Q/\mathfrak{a}$, then we write $\Gr_{\mathbf{e}}^\circ(M)$ for the set of points $N\in\Gr_{\mathbf{e}}(M)$ such that $\Ext_{\mathbb{C}Q}^1(N,M/N)=0$. If $Q$ is affine, then $\Gr_{\mathbf{e}}^\circ(M)$ is called the \emph{transverse quiver Grassmannian}.
\end{definition}

\begin{lemma}
The set $\Gr_{\mathbf{e}}^\circ(M)$ is constructible in $\Gr_{\mathbf{e}}(M)$.
\end{lemma}

\begin{proof}
The same argument used in Section~4 of~\cite{CerulliDupontEsposito} shows that for any finite-dimensional representation $M$, the set $\Gr_{\mathbf{e}}^\circ(M)$ coincides with the set of smooth points in~$\Gr_{\mathbf{e}}(M)$ contained in irreducible components of the minimal possible dimension. In particular, it is a constructible set.
\end{proof}

\subsection{Band modules}

Let $\Sigma$ be a marked bordered surface without punctures and $T$ an ideal triangulation of~$\Sigma$. Suppose $\ell$ is an integral lamination on $\Sigma$ consisting of a single closed curve $c$ of weight~$k$. We can choose an orientation for this curve $c$ and assume that it intersects each edge of the triangulation transversely in the minimal number of points. By considering the arcs of the triangulation that $c$ intersects in order as we travel around the curve according to the orientation, we obtain in an obvious way a band $w=w(\ell)$ for the quiver~$Q(T)$.

Now for any choice of $\lambda\in\mathbb{C}^*$, the construction described above gives a corresponding band module $M(w,k,\lambda)$ over the Jacobian algebra. We will be interested in varieties parametrizing submodules of this band module. Our results will turn out to be independent of~$\lambda$, and so we fix this parameter once and for all.

\begin{definition}
We write $M_{T,\ell}$ for the band module associated to the lamination~$\ell$ and a fixed choice of $\lambda\in\mathbb{C}^*$ by the above construction.
\end{definition}

We will first consider the case where $\Sigma$ is an annulus and $\ell$ consists of a single closed curve $c$ of weight $k=1$ that winds once around this annulus.

\begin{lemma}
\label{lem:computeEulercharacteristic}
Suppose $\Sigma$ is an annulus and $\ell$ is the integral lamination consisting of a single closed curve $c$ of weight~1 that winds once around the annulus. Let $\mathbf{e}$ be a dimension vector such that $e_i\leq\dim_\mathbb{C} M_{T,\ell}(i)$ for every vertex $i$ of $Q(T)$. Then we have $\chi(\Gr_{\mathbf{e}}^\circ(M_{T,\ell}))=0$ if there exists an arrow $a:i\rightarrow j$ in~$Q(T)$ such that 
\begin{enumerate}
\item $\dim_\mathbb{C} M_{T,\ell}(i)=1$ and $\dim_\mathbb{C} M_{T,\ell}(j)=1$.
\item $e_i=1$ and $e_j=0$.
\end{enumerate}
Otherwise, we have $\chi(\Gr_{\mathbf{e}}^\circ(M_{T,\ell}))=1$.
\end{lemma}

\begin{proof}
It follows from our assumptions that for each vertex $i$ of~$Q(T)$, we can have either $M_{T,\ell}(i)=\mathbb{C}$ or $M_{T,\ell}(i)=0$. If $a:i\rightarrow j$ is an arrow in~$Q(T)$ connecting vertices $i$ and~$j$, then either $M_{T,\ell}(a)$ is an isomorphism or $M_{T,\ell}(a):0\rightarrow\mathbb{C}$ or $M_{T,\ell}(a):\mathbb{C}\rightarrow0$ or $M_{T,\ell}(a):0\rightarrow0$. In the latter three cases, $M_{T,\ell}(a)$ restricts uniquely to $0\rightarrow0$, whereas the in first case there is a restriction to $0\rightarrow\mathbb{C}$ but not to $\mathbb{C}\rightarrow0$. It follows that $\Gr_{\mathbf{e}}(M_{T,\ell})$ is a single point, except when there is an arrow satisfying properties~1 and~2, in which case we have $\Gr_{\mathbf{e}}(M_{T,\ell})=\emptyset$. By Corollary~1.2 of~\cite{CerulliDupontEsposito}, the transverse quiver Grassmannian $\Gr_{\mathbf{e}}^\circ(M_{T,\ell})$ coincides with the set of smooth points contained in irreducible components of $\Gr_{\mathbf{e}}(M_{T,\ell})$ of dimension~0. Thus, under our assumptions, $\Gr_{\mathbf{e}}^\circ(M_{T,\ell})=\Gr_{\mathbf{e}}(M_{T,\ell})$.
\end{proof}

Suppose $\ell$ is a lamination consisting of a single closed curve $c$ of weight~1 and this curve intersects each edge of the triangulation transversely in the minimal number of points. In Lemma~\ref{lem:nonperipheralFpolynomial}, we saw that $\mathbb{I}(\ell)$ can be expressed as 
\[
\mathbb{I}(\ell)=F_\ell(X_1,\dots,X_n)\cdot X_1^{h_{\ell,1}}\dots X_n^{h_{\ell,n}}
\]
for a certain polynomial $F(X_1,\dots,X_n)$. To define this polynomial, we choose an orientation for the curve $c$ and consider the arcs $i_1,\dots,i_s$ that the curve $c$ passes through, ordered according to the chosen orientation. If $c$ turns to the left after crossing $i_k$, we define 
\[
S_k=
\left(
\begin{array}{cc}
X_{i_k} & X_{i_k} \\
0 & 1
\end{array}
\right),
\]
and if $c$ turns to the right, we define 
\[
S_k=
\left(
\begin{array}{cc}
X_{i_k} & 0 \\
1 & 1
\end{array}
\right).
\]
Then the polynomial $F(X_1,\dots,X_n)$ is defined as the trace of the product $S_1\dots S_s$.

Note that the above matrices factor as 
\[
\left(
\begin{array}{cc}
X_{i_k} & X_{i_k} \\
0  & 1 
\end{array}
\right)
=
\left(
\begin{array}{cc}
X_{i_k} & 0 \\
0  & 1 
\end{array}
\right)
\left(
\begin{array}{cc}
1 & 1 \\
0  & 1 
\end{array}
\right)
\]
and
\[
\left(
\begin{array}{cc}
X_{i_k} & 0 \\
1  & 1 
\end{array}
\right)
=
\left(
\begin{array}{cc}
X_{i_k} & 0 \\
0  & 1 
\end{array}
\right)
\left(
\begin{array}{cc}
1 & 0 \\
1  & 1 
\end{array}
\right).
\]
Let us denote the first matrix factor on the right hand side of either of these factorizations by~$P_k$ and denote the second factor by $Q_k$. Thus, for either form of the matrix $S_k=P_kQ_k$, we can write 
\[
Q_k=
\left(
\begin{array}{cc}
q_{11}^{(k)} & q_{10}^{(k)} \\
q_{01}^{(k)}  & q_{00}^{(k)} 
\end{array}
\right)
\]
with $q_{ij}^{(k)}=0$ or~1 for all indices $i$, $j$, and~$k$.

To establish a relationship between the $\Gr_{\mathbf{e}}^\circ(M_{T,\ell})$ and $\mathbb{I}(\ell)$, we employ a version of the state-sum description of $\mathbb{I}(\ell)$ given in~\cite{BonahonWong}.

\begin{definition}
A \emph{state} for $c$ is an assignment of numbers $\sigma_1,\dots,\sigma_s\in\{0,1\}$ to the points of intersection between $c$ and the arcs $i_1,\dots,i_s$ in order.
\end{definition}

\begin{lemma}
\label{lem:statesum}
We have 
\[
F_\ell(X_1,\dots,X_n)=\sum_\sigma q_{\sigma_1\sigma_2}^{(1)}q_{\sigma_2\sigma_3}^{(2)}\dots q_{\sigma_s\sigma_1}^{(s)}X_{i_1}^{\sigma_1}X_{i_2}^{\sigma_2}\dots X_{i_s}^{\sigma_s}
\]
where the sum is over all states $\sigma$.
\end{lemma}

\begin{proof}
In the above notation, we have 
\[
S_k=P_kQ_k=
\left(
\begin{array}{cc}
X_{i_k}q_{11}^{(k)} & X_{i_k}q_{10}^{(k)} \\
q_{01}^{(k)}  & q_{00}^{(k)} 
\end{array}
\right).
\]
Write $(S_k)_{ij}$ for the $(i,j)$ entry of this matrix where we consider the indices modulo~2. Then 
\begin{align*}
\Tr(S_1\dots S_s) &= \sum_{\sigma_1,\dots,\sigma_s\in\{0,1\}}(S_1)_{\sigma_1\sigma_2}(S_2)_{\sigma_2\sigma_3}\dots(S_s)_{\sigma_s\sigma_1} \\
&= \sum_{\sigma_1,\dots,\sigma_s\in\{0,1\}} q_{\sigma_1\sigma_2}^{(1)}X_{i_1}^{\sigma_1}\cdot q_{\sigma_2\sigma_3}^{(2)}X_{i_2}^{\sigma_2}\cdot\dots\cdot q_{\sigma_s\sigma_1}^{(s)}X_{i_s}^{\sigma_s}
\end{align*}
as desired.
\end{proof}

From this state-sum formula, we obtain the following.

\begin{lemma}
\label{lem:simplebandGrassmannian}
Suppose $\Sigma$ is an annulus and $\ell$ is the integral lamination consisting of a single closed curve $c$ of weight~1 that winds once around the annulus. Then 
\[
F_\ell(X_1,\dots,X_n)=\sum_{\mathbf{e}=(e_1,\dots,e_n)\in\mathbb{Z}_{\geq0}^n}\chi(\Gr_{\mathbf{e}}^\circ(M_{T,\ell}))X_1^{e_1}\dots X_n^{e_n}.
\]
\end{lemma}

\begin{proof}
Consider a dimension vector $\mathbf{e}$ satisfying $e_i\leq\dim_\mathbb{C}M_{T,\ell}(i)$ for every vertex $i$ of~$Q(T)$. If the Euler characteristic $\chi(\Gr_{\mathbf{e}}^\circ(M_{T,\ell}))$ vanishes, then by Lemma~\ref{lem:computeEulercharacteristic}, there exists an arrow $a:i\rightarrow j$ in~$Q(T)$ such that 
\begin{enumerate}
\item $\dim_\mathbb{C} M_{T,\ell}(i)=1$ and $\dim_\mathbb{C} M_{T,\ell}(j)=1$.
\item $e_i=1$ and $e_j=0$.
\end{enumerate}
The first property implies that the curve $c$ connects the arcs $i$ and $j$ within some triangle~$t$ of the triangulation~$T$. If the curve $c$ is oriented from~$i=i_k$ to~$j=i_{k+1}$, then it turns to the right as it passes over the triangle $t$, so we have 
\[
Q_k=
\left(
\begin{array}{cc}
1 & 0 \\
1 & 1
\end{array}
\right),
\]
and hence $q_{10}^{(k)}=0$. Similarly, if the curve is oriented from $j=i_k$ to $i=i_{k+1}$, then it turns to the left, so we have 
\[
Q_k=
\left(
\begin{array}{cc}
1 & 1 \\
0 & 1
\end{array}
\right),
\]
hence $q_{01}^{(k)}=0$. In either case, it follows from Lemma~\ref{lem:statesum} that the coefficient of $X_1^{e_1}\dots X_n^{e_n}$ in the Laurent expansion of $F(X_1,\dots,X_n)$ vanishes. Conversely, if the coefficient of $X_1^{e_1}\dots X_n^{e_n}$ vanishes, then we have $q_{\sigma_k\sigma_{k+1}}^{(k)}=0$ for some~$k$. This means either $c$ turns to the right after crossing~$i_k$ and we have $\sigma_k=1$ and $\sigma_{k+1}=0$, or $c$ turns to the left and we have $\sigma_k=0$ and $\sigma_{k+1}=1$. It follows that there is an arrow satisfying properties~1 and~2, and $\chi(\Gr_{\mathbf{e}}^\circ(M_{T,\ell}))=0$. Finally, if $\chi(\Gr_{\mathbf{e}}^\circ(M_{T,\ell}))$ is nonzero, then we must have $\chi(\Gr_{\mathbf{e}}^\circ(M_{T,\ell}))=1$, and similarly any nonzero term in the state-sum expression for $F(X_1,\dots,X_n)$ has coefficient~1. This completes the proof.
\end{proof}

Given an integral lamination $\ell$ and an integer $k\geq1$, let us write $k\ell$ for the lamination obtained from~$\ell$ by multiplying the weight of each curve by~$k$. The following lemma is the reason we work with transverse rather than ordinary quiver Grassmannians.

\begin{lemma}
\label{lem:bandGrassmannian}
Suppose $\Sigma$ is an annulus and $\ell$ is the integral lamination consisting of a single closed curve $c$ of weight~1 that winds once around the annulus. Then for any $k\geq1$, we have 
\[
F_{k\ell}(X_1,\dots,X_n)=\sum_{\mathbf{e}=(e_1,\dots,e_n)\in\mathbb{Z}_{\geq0}^n}\chi(\Gr_{\mathbf{e}}^\circ(M_{T,k\ell}))X_1^{e_1}\dots X_n^{e_n}.
\]
\end{lemma}

\begin{proof}
To any indecomposable module~$M$ over $J(Q(T),W(T))$, Dupont~\cite{Dupont} associates a Laurent polynomial in the cluster variables by the formula 
\[
CC^\circ(M)=\sum_{\mathbf{e}}\chi(\Gr_{\mathbf{e}}^\circ(M))\prod_{j=1}^nX_j^{e_j} \left(A_j^{-1}\prod_{i=1}^nA_i^{[\varepsilon_{ij}]_+}\right)^{\dim M_j}
\]
where $[r]_{+}=\max(r,0)$. This expression is a variant of the well known \emph{cluster character} in which ordinary quiver Grassmannians have been replaced by transverse quiver Grassmannians. One can check that 
\[
\left(\prod_{j=1}^nA_j^{-1}\prod_{i=1}^nA_i^{[\varepsilon_{ij}]_+}\right)^2=\prod_{j=1}^nX_j^{-1},
\]
and hence if $\ell$ consists of a single closed curve $c$ of weight~1, then 
\[
CC^\circ(M_{T,\ell})=F_\ell(X_1,\dots,X_n)X_1^{-1/2}\dots X_n^{-1/2}
\]
for some choice of square roots. Thus $CC^\circ(M_{T,\ell})$ equals the trace of a matrix in~$SL_2$ representing the monodromy around the curve~$c$. By the main result of~\cite{Dupont}, we have 
\[
CC^\circ(M_{T,k\ell})=P_k(CC^\circ(M_{T,\ell}))
\]
where $P_k$ is the $k$th Chebyshev polynomial. By~\cite{AllegrettiKim}, Proposition~3.6, this last expression equals the trace of the $k$th power of the monodromy matrix. By what we have said, there is a factorization 
\[
CC^\circ(M_{T,k\ell})=\left(\sum_{\mathbf{e}}\chi(\Gr_{\mathbf{e}}^\circ(M_{T,k\ell}))\prod_{j=1}^nX_j^{e_j}\right)X_1^{-k/2}\dots X_n^{-k/2}.
\]
Comparing with the expression for $\mathbb{I}(k\ell)$ given by Lemma~\ref{lem:nonperipheralFpolynomial}, we obtain the desired result.
\end{proof}

Finally, we consider the case where $\Sigma$ is an arbitrary marked bordered surface without punctures and $\ell$ consists of a single closed curve $c$ of arbitrary weight~$k$. Let $t_0$ be any triangle that $c$ intersects, and let $\tilde{t}_0$ be a copy of $t_0$ with an orientation preserving homeomorphism $\tilde{t}_0\rightarrow t_0$. If we choose an orientation for $c$, then we can travel along this curve in the direction specified by the orientation, starting from some point in $t_0$, and we eventually enter a neighboring triangle~$t_1$. Let $\tilde{t}_1$ be a copy of~$t_1$ with an orientation preserving homeomorphism $\tilde{t}_1\rightarrow t_1$, and let us glue $\tilde{t}_0$ and $\tilde{t}_1$ along a boundary edge so that the homeomorphisms agree on the glued edge. If we continue traveling along~$c$, we enter another triangle $t_2$, and we can glue a homeomorphic copy $\tilde{t}_2$ of this triangle to the union $\tilde{t}_0\cup\tilde{t}_1$. Continue this process until we return to the initial triangle for the last time. At this point, we have constructed a surface 
\[
\tilde{t}_0\cup\tilde{t}_1\cup\dots\cup\tilde{t}_N,
\]
and we now glue the triangles $\tilde{t}_N$ and $\tilde{t}_0$ in this expression along an edge so that the homeomorphisms $\tilde{t}_N\rightarrow t_N$ and $\tilde{t}_0\rightarrow t_0$ agree on the glued edge. In this way, we obtain a surface homeomorphic to an annulus.

The annulus defined in this way is naturally a marked bordered surface equipped with an ideal triangulation $\widetilde{T}$ whose triangles are $\tilde{t}_0,\dots,\tilde{t}_N$. The curve $c$ can be lifted to a closed curve on this surface which we regard as an integral lamination $\tilde{\ell}$ by assigning it the weight~$k$. We also get a representation $M_{\widetilde{T},\tilde{\ell}}$ of $Q(\widetilde{T})$ by the construction described previously. Note that there is a map of quivers $f:Q(\widetilde{T})\rightarrow Q(T)$. It sends each arc of $\widetilde{T}$ to the corresponding arc of $T$.

\begin{lemma}
\label{lem:coveringformula}
The Euler characteristic of $\Gr_{\mathbf{e}}^\circ(M_{T,\ell})$ is 
\[
\chi(\Gr_{\mathbf{e}}^\circ(M_{T,\ell}))=\sum_{\mathbf{d}}\chi(\Gr_{\mathbf{d}}^\circ(M_{\widetilde{T},\tilde{\ell}}))
\]
where the sum is over all dimension vectors for $Q(\widetilde{T})$ such that $e_i=\sum_{\tilde{i}\in f^{-1}(i)}d_{\tilde{i}}$.
\end{lemma}

\begin{proof}
For convenience, let us write $M=M_{T,\ell}$ and $\widetilde{M}=M_{\widetilde{T},\tilde{\ell}}$. Then we have $M=f_*(\widetilde{M})$ by construction. We claim that if $\widetilde{N}\in\Gr_{\mathbf{d}}(\widetilde{M})$ is a submodule and $N=f_*(\widetilde{N})\in\Gr_{\mathbf{e}}(M)$, then $\widetilde{N}\in\Gr_{\mathbf{d}}^\circ(\widetilde{M})$ if and only if $N\in\Gr_{\mathbf{e}}^\circ(M)$. To see this, suppose we are given that $N\in\Gr_{\mathbf{e}}^\circ(M)$. If 
\begin{align}
\label{eqn:extensionupstairs}
\xymatrix{ 
0 \ar[r] & \widetilde{N} \ar[r] & \widetilde{V} \ar[r] & \widetilde{M}/\widetilde{N} \ar[r] & 0
}
\end{align}
is an extension of $\mathbb{C}Q(\widetilde{T})$-modules, then by applying $f_*$, we get a corresponding extension 
\begin{align}
\label{eqn:extensiondownstairs}
 \xymatrix{ 
0 \ar[r] & N \ar[r]^-{\iota} & V \ar[r]^-{\pi} & M/N \ar[r] & 0
}
\end{align}
of $\mathbb{C}Q(T)$-modules which splits by our assumption on $N$. It follows that~\eqref{eqn:extensionupstairs} splits and we have $\widetilde{N}\in\Gr_{\mathbf{d}}^\circ(\widetilde{M})$.

Conversely, suppose we are given that $\widetilde{N}\in\Gr_{\mathbf{d}}^\circ(\widetilde{M})$. Suppose we have an extension of the form~\eqref{eqn:extensiondownstairs}. For any vertex $i$ of the quiver $Q(T)$, we have decompositions $M(i)=\bigoplus_{\tilde{i}\in f^{-1}(i)}\widetilde{M}(\tilde{i})$, \ $N(i)=\bigoplus_{\tilde{i}\in f^{-1}(i)}\widetilde{N}(\tilde{i})$, and $(M/N)(i)=\bigoplus_{\tilde{i}\in f^{-1}(i)}\widetilde{M}(\tilde{i})/\widetilde{N}(\tilde{i})$. For each vertex $\tilde{i}$ of $Q(\widetilde{T})$, let us write $\widetilde{V}(\tilde{i})$ for the preimage of $\widetilde{M}(\tilde{i})/\widetilde{N}(\tilde{i})$ under the map~$\pi$. One can check that any two of the vector spaces $\widetilde{V}(\tilde{i})$ associated to the vertices $\tilde{i}$ in this way intersect trivially and $V=\sum_{\tilde{i}\in f^{-1}(i)}\widetilde{V}(\tilde{i})$. Hence for each $i$ we have a decomposition 
\[
V(i)\cong\bigoplus_{\tilde{i}\in f^{-1}(i)}\widetilde{V}(\tilde{i}).
\]
Moreover, for any arrow $a$ in $Q(T)$, the map $V(a)$, respects the decompositions so that we have a representation $\widetilde{V}$ of $Q(\widetilde{T})$ and $V=f_*(\widetilde{V})$. We then have a split exact sequence of the form~\eqref{eqn:extensionupstairs}, and hence \eqref{eqn:extensiondownstairs} is split. It follows that $N\in\Gr_{\mathbf{e}}^\circ(M)$. This proves our claim.

Now let us consider the map 
\[
f_*:\coprod_{\mathbf{d}}\Gr_{\mathbf{d}}(\widetilde{M})\rightarrow\Gr_{\mathbf{e}}(M)
\]
where the union is over all dimension vectors for $Q(\widetilde{T})$ such that $e_i=\sum_{\tilde{i}\in f^{-1}(i)}d_{\tilde{i}}$. In~\cite{Haupt}, Haupt showed that there exists an action of an algebraic torus $\mathbb{T}$ on~$\Gr_{\mathbf{e}}(M)$ such that the map $f_*$ identifies this union with the set of $\mathbb{T}$-fixed points. Thus, by the above claim, there is a decomposition of $\Gr_{\mathbf{e}}^\circ(M)^{\mathbb{T}}$ as a union of constructible subsets 
\[
\Gr_{\mathbf{e}}^\circ(M)^{\mathbb{T}}=\coprod_{\mathbf{d}}\Gr_{\mathbf{d}}^\circ(\widetilde{M}).
\]
Since $\Gr_{\mathbf{e}}^\circ(M)$ is a constructible subset of $\Gr_{\mathbf{e}}(M)$, Proposition~5.1 of~\cite{Haupt} implies the equality of Euler characteristics $\chi(\Gr_{\mathbf{e}}^\circ(M))=\chi(\Gr_{\mathbf{e}}^\circ(M)^{\mathbb{T}})$. The lemma now follows by the excision property of the Euler characteristic.
\end{proof}

We can now prove the desired generalization of Lemma~\ref{lem:bandGrassmannian}.

\begin{proposition}
\label{prop:Fpolynomialbandmodule}
Let $\ell$ be an integral lamination consisting of a single closed curve $c$ of weight~$k$. Then 
\[
F_\ell(X_1,\dots,X_n)=\sum_{\mathbf{e}=(e_1,\dots,e_n)\in\mathbb{Z}_{\geq0}^n}\chi(\Gr_{\mathbf{e}}^\circ(M_{T,\ell}))X_1^{e_1}\dots X_n^{e_n}.
\]
\end{proposition}

\begin{proof}
We have seen how to construct an auxiliary marked bordered surface, which is topologically an annulus, equipped with an ideal triangulation~$\widetilde{T}$. The lamination $\ell$ can be lifted to a lamination $\tilde{\ell}$ on this annulus consisting of a single closed curve $\tilde{c}$ of weight~$k$, and we can assume that $\tilde{c}$ passes through each triangle of $\widetilde{T}$ exactly once. Let us assign a variable $X_{\tilde{i}}$ to each arc~$\tilde{i}$ of~$\widetilde{T}$. Then, by taking a trace as described above, we can associate a polynomial $F_{\tilde{\ell}}$ in these variables to~$\tilde{\ell}$. By Lemma~\ref{lem:bandGrassmannian}, this polynomial is given by 
\[
F_{\tilde{\ell}}=\sum_{\mathbf{d}}\chi(\Gr_{\mathbf{d}}^\circ(M_{\widetilde{T},\tilde{\ell}}))\prod_{\tilde{i}}X_{\tilde{i}}^{d_{\tilde{i}}}
\]
where the sum runs over all dimension vectors for the quiver $Q(\widetilde{T})$. To recover the polynomial $F_\ell(X_1,\dots,X_n)$, we simply replace each variable $X_{\tilde{i}}$ in the above expression for $F_{\tilde{\ell}}$ by $X_i$ where $f(\tilde{i})=i$. Thus the result follows from Lemma~\ref{lem:coveringformula}.
\end{proof}

\subsection{String modules}

Let $\Sigma$ be a marked bordered surface without punctures and $T$ an ideal triangulation of~$\Sigma$. Suppose $\gamma$ is an arc on $\Sigma$ such that the interior of $\gamma$ intersects the edges of $T$ transversely and the number of intersections of $\gamma$ with a given edge equals the geometric intersection number. Choose an orientation for this arc~$\gamma$. By considering the arcs of the triangulation that $\gamma$ intersects in order as we travel along this arc according to the orientation, we obtain in an obvious way, a string for the quiver~$Q(T)$.

In particular, by the construction described above, we define a string module over the Jacobian algebra.

\begin{definition}
We write $M_{T,\gamma}$ for the string module associated to the arc~$\gamma$ by the above construction.
\end{definition}

Building on earlier work of Derksen, Weyman, and Zelevinsky~\cite{DWZ2} on representations of Jacobian algebras, Labardini-Fragoso proved the following.

\begin{theorem}[\cite{L2}, Corollary~6.7]
\label{thm:Fpolynomialstringmodule}
The $F$-polynomial $F_\gamma$ associated to the arc $\gamma$ and the initial triangulation $T$ is given by the formula 
\[
F_\gamma(X_1,\dots,X_n)=\sum_{\mathbf{e}=(e_1,\dots,e_n)\in\mathbb{Z}_{\geq0}^n}\chi(\Gr_{\mathbf{e}}M_{T,\gamma})X_1^{e_1}\dots X_n^{e_n}.
\]
\end{theorem}

\subsection{Categorification of canonical bases}

In this subsection, we will use the above results to construct graded vector spaces which categorify the canonical basis. This will yield a proof of Theorem~\ref{thm:introcategorification1}.

Consider a lamination $\ell\in\mathcal{A}(\mathbb{Z}^t)$ on the marked bordered surface $\Sigma$. Such a lamination can be represented by a collection of curves on~$\mathbb{S}$ such that there is at most one closed curve in any homotopy class and any open curve has weight $\pm1$. If a curve ends on a boundary segment, let us modify this curve by dragging its endpoints along the boundary in the counterclockwise direction until they hit the marked points. In this way we obtain a collection of arcs, boundary segments, and closed loops, and we can assume that these intersect the arcs of the triangulation transversely in the minimal number of points.

Let $\{\gamma_1,\dots,\gamma_r\}$ be the set of arcs obtained in this way, and let $\{\ell_1,\dots,\ell_s\}$ be the set of laminations that we get from the individual closed curves of~$\ell$. Let 
\[
M_{T,\ell}=\bigoplus_{i=1}^rM_{T,\gamma_i}\oplus\bigoplus_{j=1}^sM_{T,\ell_j}
\]
be the direct sum of the string and band modules defined above. For each loop~$\ell_j$, let us write $\pi_j:M_{T,\ell}\rightarrow M_{T,\ell_j}$ for the projection onto~$M_{T,\ell_j}$. We will consider the following space parametrizing submodules of $M_{T,\ell}$.

\begin{definition}
We write $\Gr_{\mathbf{d}}^\circ(M_{T,\ell})$ for the set of points $N\in\Gr_{\mathbf{d}}(M_{T,\ell})$ such that, for any loop $\ell_j$ of the lamination~$\ell$, we have $\pi_j(N)\in\Gr_{\mathbf{e}}^\circ(M_{T,\ell_j})$ for some dimension vector~$\mathbf{e}$.
\end{definition}

\begin{lemma}
The set $\Gr_{\mathbf{d}}^\circ(M_{T,\ell})$ is a constructible subset of $\Gr_{\mathbf{d}}(M_{T,\ell})$.
\end{lemma}

\begin{proof}
For any index $j$, there is a decomposition $M_{T,\ell}=R\oplus M_{T,\ell_j}$ where $R$ is the direct sum of the string and band modules not isomorphic to $M_{T,\ell_j}$. By setting 
\[
\Phi_j(N)=(N\cap R,\pi_j(N)),
\]
we obtain a constructible map 
\[
\Phi_j:\Gr_{\mathbf{d}}(R\oplus M_{T,\ell_j})\rightarrow\coprod_{\mathbf{e}_1+\mathbf{e}_2=\mathbf{d}}\Gr_{\mathbf{e}_1}(R)\times\Gr_{\mathbf{e}_2}(M_{T,\ell_j}).
\]
For any dimension vectors $\mathbf{e}_1$ and~$\mathbf{e}_2$ such that $\mathbf{e}_1+\mathbf{e}_2=\mathbf{d}$, the product $\Gr_{\mathbf{e}_1}(R)\times\Gr_{\mathbf{e}_2}^\circ(M_{T,\ell_j})$ is a constructible subset of the right hand side, and hence the preimage 
\[
C_j\coloneqq\Phi_j^{-1}\left(\coprod_{\mathbf{e}_1+\mathbf{e}_2=\mathbf{d}}\Gr_{\mathbf{e}_1}(R)\times\Gr_{\mathbf{e}_2}^\circ(M_{T,\ell_j})\right)
\]
is constructible. Hence 
\[
\Gr_{\mathbf{d}}^\circ(M_{T,\ell})=\bigcap_{j=1}^sC_j
\]
is constructible as required.
\end{proof}

We will attach to the lamination~$\ell$ the $\mathbb{Z}^n\times\mathbb{Z}$-graded vector space 
\[
\mathcal{H}^{T,\ell}\coloneqq\bigoplus_{\mathbf{d},i}\mathcal{H}_{\mathbf{d},i}^{T,\ell}
\]
where 
\[
\mathcal{H}_{\mathbf{d},i}^{T,\ell}\coloneqq H^i(\Gr_{\mathbf{d}-\mathbf{h}}^\circ(M_{T,\ell}),\mathbb{C})
\]
and $\mathbf{h}=(h_{\ell,1},\dots,h_{\ell,n})$ is the vector from Proposition~\ref{prop:generalFpolynomial}. To calculate the graded dimension of this vector space, we employ the following result on Euler characteristics.

\begin{lemma}
\label{lem:eulercharacteristicsum}
Let $\mathbf{d}$ be a dimension vector for representations of~$Q(T)$. Then 
\[
\chi(\Gr_{\mathbf{d}}^\circ(M_{T,\ell}))=\sum_{\stackrel{\mathbf{e}_1,\dots,\mathbf{e}_r}{\mathbf{f}_1,\dots,\mathbf{f}_s}}\left(\prod_{i=1}^r\chi(\Gr_{\mathbf{e}_i}(M_{T,\gamma_i}))\cdot\prod_{j=1}^s\chi(\Gr_{\mathbf{f}_j}^\circ(M_{T,\ell_j}))\right)
\]
where the sum is over all dimension vectors $\mathbf{e}_i$ and $\mathbf{f}_j$ such that $\mathbf{d}=\sum_i\mathbf{e}_i+\sum_j\mathbf{f}_j$.
\end{lemma}

\begin{proof}
Note that if $M=M_1\oplus\dots\oplus M_p$ is a direct sum of modules over the Jacobian algebra, then we can define an action of the algebraic torus $\mathbb{T}=(\mathbb{C}^*)^{p-1}$ on $M$ by 
\[
(t_1,\dots,t_{p-1})\cdot(m_1,\dots,m_p)=(t_1m_1,\dots,t_{p-1}m_{p-1},m_p),
\]
and this induces an algebraic action of $\mathbb{T}$ on the variety $\Gr_{\mathbf{d}}(M)$. Moreover, the fixed points of $\Gr_{\mathbf{d}}(M)$ under this action are precisely the submodules $N\subseteq M$ possessing a direct sum decomposition $N=N_1\oplus\dots\oplus N_p$ where $N_i\subseteq M_i$ is a submodule for each~$i$. In particular, we can apply this to the module $M=M_{T,\ell}$ constructed above, which is a direct sum of string and band modules. In this case, the set of $\mathbb{T}$-fixed points in $\Gr_{\mathbf{d}}^\circ(M_{T,\ell})$ is 
\[
\Gr_{\mathbf{d}}^\circ(M_{T,\ell})^{\mathbb{T}}=\coprod_{\stackrel{\mathbf{e}_1,\dots,\mathbf{e}_r}{\mathbf{f}_1,\dots,\mathbf{f}_s}}\left(\prod_{i=1}^r\Gr_{\mathbf{e}_i}(M_{T,\gamma_i})\times\prod_{j=1}^s\Gr_{\mathbf{f}_j}^\circ(M_{T,\ell_j})\right)
\]
where the sum is over all dimension vectors $\mathbf{e}_i$ and $\mathbf{f}_j$ such that $\mathbf{d}=\sum_i\mathbf{e}_i+\sum_j\mathbf{f}_j$. Since $\Gr_{\mathbf{d}}^\circ(M_{T,\ell})$ is a constructible set in $\Gr_{\mathbf{d}}(M_{T,\ell})$, Proposition~5.1 of~\cite{Haupt} implies the equality of Euler characteristics $\chi(\Gr_{\mathbf{d}}^\circ(M_{T,\ell}))=\chi(\Gr_{\mathbf{d}}^\circ(M_{T,\ell})^{\mathbb{T}})$. The lemma follows from this and the above formula for the set of fixed points.
\end{proof}

The following is one of the main results of this paper. For convenience, we will write $X_{\mathbf{d}}=\prod_jX_j^{d_j}$ for any dimension vector~$\mathbf{d}$.

\begin{theorem}
\label{thm:categorifycanonicalbasis}
For any integral lamination $\ell\in\mathcal{A}(\mathbb{Z}^t)$, we have 
\[
\mathbb{I}(\ell)=\sum_{\mathbf{d}\in\mathbb{Z}^n,i\in\mathbb{Z}}(-1)^i\dim_{\mathbb{C}}\mathcal{H}_{\mathbf{d},i}^{T,\ell}\cdot X_{\mathbf{d}}.
\]
\end{theorem}

\begin{proof}
By definition of the space $\mathcal{H}_{\mathbf{d},i}^{T,\ell}$, we have 
\begin{align*}
\sum_{\mathbf{d},i}(-1)^i\dim_{\mathbb{C}}\mathcal{H}_{\mathbf{d},i}^{T,\ell}\cdot X_{\mathbf{d}} &= \sum_{\mathbf{d}}\chi(\Gr_{\mathbf{d}-\mathbf{h}}^\circ(M_{T,\ell}))X_{\mathbf{d}} \\
&= \left(\sum_{\mathbf{d}}\chi(\Gr_{\mathbf{d}}^\circ(M_{T,\ell}))X_{\mathbf{d}}\right)X_{\mathbf{h}}.
\end{align*}
By Lemma~\ref{lem:eulercharacteristicsum}, Theorem~\ref{thm:Fpolynomialstringmodule}, and Proposition~\ref{prop:Fpolynomialbandmodule}, the expression in parentheses is 
\begin{align*}
\sum_{\mathbf{d}} & \chi(\Gr_{\mathbf{d}}^\circ(M_{T,\ell}))\cdot X_{\mathbf{d}} \\ 
&= \prod_{i=1}^r\left(\sum_{\mathbf{e}}\chi(\Gr_{\mathbf{e}}(M_{T,\gamma_i}))X_{\mathbf{e}}\right)
\cdot \prod_{j=1}^s\left(\sum_{\mathbf{f}}\chi(\Gr_{\mathbf{f}}^\circ(M_{T,\ell_j}))X_{\mathbf{f}}\right) \\
&= \prod_{i=1}^r F_{\gamma_i}(X_1,\dots,X_n)\cdot\prod_{j=1}^s F_{\ell_j}(X_1,\dots,X_n).
\end{align*}
Substituting this back into the previous expression and comparing with Proposition~\ref{prop:generalFpolynomial}, we obtain the desired result.
\end{proof}

\section{Framed quiver moduli}
\label{sec:FramedQuiverModuli}

\subsection{Construction of framed quivers}

Given a quiver with potential $(Q,W)$, let us write $\mathcal{A}=\Mod J(Q,W)$ for the abelian category of finite-dimensional modules over the Jacobian algebra $J(Q,W)$. To make contact with the theory of framed BPS states in $\mathcal{N}=2$ field theories, we will consider certain moduli spaces of stable objects in the category~$\mathcal{A}$.

\begin{definition}
A \emph{stability condition} on the category $\mathcal{A}$ is a group homomorphism $Z:K(\mathcal{A})\rightarrow\mathbb{C}$ such that for any nonzero object $E$, the image $Z([E])$ of the class of $E$ in the Grothendieck group lies in the semi-closed upper half plane 
\[
\mathfrak{h}=\{r\exp(i\pi\phi):r>0 \text{ and } 0<\phi\leq1\}\subset\mathbb{C}.
\]
An object $M$ of $\mathcal{A}$ is said to be \emph{stable} if, for all proper subobjects $N$ of~$M$, we have 
\[
\arg Z([N])<\arg Z([M]).
\]
\end{definition}

There is a natural bijection between the isomorphism classes of simple objects in the category~$\mathcal{A}$ and vertices of the quiver~$Q$. We can therefore view the Grothendieck group $K(\mathcal{A})$ as the free abelian group on vertices of~$Q$. Then a stability condition is equivalent to an assignment of a point $Z_i\in\mathfrak{h}$ to each vertex $i\in Q_0$.

When talking about framed BPS states in theories of class~$\mathcal{S}$, the relevant moduli spaces parametrize representations of a framed version of the quiver associated to an ideal triangulation. To define this quiver, we consider a marked bordered surface $\Sigma$ and an ideal triangulation $T$ of~$\Sigma$. Given an integral lamination $\ell\in\mathcal{A}(\mathbb{Z}^t)$ on~$\Sigma$, we set 
\[
n_j\coloneqq\langle\mathbf{e}_j,\mathbf{h}\rangle
\]
where $\mathbf{h}=(h_{\ell,k})$ is the vector defined by Proposition~\ref{prop:generalFpolynomial} and $\langle\cdot,\cdot\rangle$ is the skew form on $\Gamma_T=\bigoplus_{j\in J}\mathbb{Z}\mathbf{e}_j$ defined on basis elements by 
\[
\langle\mathbf{e}_i,\mathbf{e}_j\rangle\coloneqq\varepsilon_{ij}.
\]
We can modify the quiver $Q=Q(T)$ using the $n_j$ to get a new quiver~$\widetilde{Q}$ as follows. The vertex set of this new quiver is defined as a union $\widetilde{Q}_0=Q_0\cup\{\infty\}$ of the vertices of $Q$ and a single additional vertex denoted~$\infty$. The arrows of the new quiver consist of the arrows of the original quiver $Q$, together with $n_j$ arrows from $j$ to~$\infty$ whenever $n_j\geq0$, and $-n_j$ arrows from $\infty$ to~$j$ whenever $n_j\leq0$.

\begin{definition}
The quiver $\widetilde{Q}$ defined in this way is called the \emph{framed quiver} associated to~$\ell$ and~$T$, and the additional vertex $\infty$ is called the \emph{framing vertex}.
\end{definition}

We have seen that the quiver $Q=Q(T)$ is equipped with a canonical potential $W=W(T)$. This can be extended to a potential $\widetilde{W}$ for the framed quiver $\widetilde{Q}$. It is a formal sum which includes all of the terms of $W$ regarded as cycles in the quiver $\widetilde{Q}$. Note that the framed quiver $\widetilde{Q}$ may contain new oriented cycles in addition to the ones already present in~$Q$. In this case, the potential $\widetilde{W}$ should include additional terms involving cycles that go through the framing vertex, and these terms should be chosen generically. Unfortunately, it is not known how to associate a canonical potential $\widetilde{W}$ to the framed quiver, and as shown in~\cite{CordovaNeitzke}, the quiver with potential $(\widetilde{Q},\widetilde{W})$ may not describe the correct spectrum of framed BPS states if this potential is chosen incorrectly. For the remainder of this section, we will therefore impose the following assumption.

\begin{assumption}
\label{assumption}
The lamination $\ell$ satisfies $n_j\geq0$ for every arc~$j$ of the chosen ideal triangulation.
\end{assumption}

Under this assumption, there are no additional oriented cycles in the framed quiver $\widetilde{Q}$, and hence there is no ambiguity in the choice of potential $\widetilde{W}$.

Finally, given a stability condition $Z:\Gamma_T\rightarrow\mathbb{C}$, we get a stability condition 
\[
\widetilde{Z}:\Gamma_T\oplus\mathbb{Z}\mathbf{e}_\infty\rightarrow\mathbb{C}
\]
for the framed quiver. The value of $\widetilde{Z}$ on a basis vector is given by $\widetilde{Z}(\mathbf{e}_i)=Z(\mathbf{e}_i)$ for $i\in Q_0$ and 
\[
\widetilde{Z}(\mathbf{e}_\infty)=\zeta m
\]
where $m\gg0$ and $\zeta\in U(1)$ is a phase chosen so that $\arg\widetilde{Z}(\mathbf{e}_\infty)<\arg\widetilde{Z}(\mathbf{e}_i)$ for all $i\in Q_0$. We will write $\mathcal{M}_{\mathbf{d}}^{\mathrm{st}}(\widetilde{Q})$ for the moduli space of stable modules over $J(\widetilde{Q},\widetilde{W})$ with respect to this stability condition having dimension $d_i$ at $i\in Q_0$ and having dimension~1 at~$\infty$.

For this particular choice of stability condition, we can give a simple algebraic description of the stable modules.

\begin{definition}
Let $M$ be a representation of $\widetilde{Q}$ and $V$ the representation defined by a one-dimensional vector space supported at the framing vertex. We say that $M$ is \emph{cyclic} if there is no proper subrepresentation of $M$ containing $V$. We say that $M$ is \emph{cocyclic} if all nonzero subrepresentations of $M$ contain~$V$.
\end{definition}

By our assumptions on $m$ and $\zeta$ in the definition of the stability condition $\widetilde{Z}$, we have the following characterization of stable modules for $J(\widetilde{Q},\widetilde{W})$.

\begin{proposition}
\label{prop:stablecyclic}
Let $M\in\Mod J(\widetilde{Q},\widetilde{W})$ be a module having dimension $d_i$ at $i\in Q_0$ and dimension~1 at~$\infty$. Then $M\in\mathcal{M}_{\mathbf{d}}^{\mathrm{st}}(\widetilde{Q})$ if and only if $M$ is cocyclic.
\end{proposition}

\subsection{Moduli spaces of cocyclic modules}

In this subsection, we will discuss in detail the moduli space of cocyclic representations of a framed quiver. In this discussion, we will fix a dimension vector $\mathbf{d}=(d_j)_{j\in Q_0}$ for the quiver $Q=Q(T)$ as well as an integral vector $\mathbf{n}=(n_j)_{j\in Q_0}$ with components $n_j\geq0$. For each $j\in Q_0$, we write $M_j$ (respectively,~$V_j$) for a fixed $\mathbb{C}$-vector space of dimension~$d_j$ (respectively,~$n_j$). We will consider the $Q_0$-graded vector spaces 
\[
M=\bigoplus_{j\in Q_0}M_j \quad\text{and}\quad V=\bigoplus_{j\in Q_0}V_j
\]
and $Q_0$-graded maps between them.

A $\mathbf{d}$-dimensional representation of~$Q$ can be viewed as an element of the variety 
\[
\Rep(Q,\mathbf{d})=\bigoplus_{a:i\rightarrow j}\Hom_\mathbb{C}(M_i,M_j)
\]
where the sum runs over all arrows in~$Q$. By Theorem~\ref{thm:Jfinitedimensional}, the Jacobian algebra $J=J(Q,W)$ can be written as a quotient $J=\mathbb{C}Q/\mathfrak{a}$ of the path algebra by an ideal~$\mathfrak{a}$. We will write $\Rep(J,\mathbf{d})$ for the subvariety of $\Rep(Q,\mathbf{d})$ consisting of points which, when viewed as modules over the path algebra, are annihilated by the ideal $\mathfrak{a}$. Note that the group 
\[
G_{\mathbf{d}}=\prod_{j\in Q_0}GL(M_j)
\]
acts naturally on the varieties $\Rep(Q,\mathbf{d})$ and $\Rep(J,\mathbf{d})$ by 
\[
(g_j)_{j\in Q_0}\cdot(M_a)_{a:i\rightarrow j}=(g_jM_ag_i^{-1})_{a:i\rightarrow j}.
\]
The orbits of this group action are identified with isomorphism classes of representations with dimension vector $\mathbf{d}$.

As we have seen, the integers $n_j$ can be used to construct a framed quiver $\widetilde{Q}$ with vertex set $Q_0\cup\{\infty\}$ and $n_j$ arrows $j\rightarrow\infty$. A $\mathbf{d}$-dimensional representation of this framed quiver~$\widetilde{Q}$ is a point of the variety 
\begin{align*}
\Rep_{\mathrm{fr}}(Q,\mathbf{d},\mathbf{n}) &= \Rep(Q,\mathbf{d})\times\bigoplus_j\Hom_\mathbb{C}(\mathbb{C}^{d_j},\mathbb{C})^{n_j} \\
&\cong \Rep(Q,\mathbf{d})\times\bigoplus_j\Hom_\mathbb{C}(M_j,V_j).
\end{align*}
We will write $\Rep_{\mathrm{fr}}(J,\mathbf{d},\mathbf{n})$ for the subvariety consisting of pairs $(M,f)$ where $M\in\Rep(Q,\mathbf{d})$ is a representation that lies in $\Rep(J,\mathbf{d})$ and $f=(f_j)_{j\in Q_0}$ is a map $M\rightarrow V$ of $Q_0$-graded vector spaces. Such a pair corresponds to a cocyclic module over the Jacobian algebra if there is no nonzero subrepresentation of $M$ contained in the kernel of $f$. We will write $\Rep_{\mathrm{fr}}^{\mathrm{cc}}(J,\mathbf{d},\mathbf{n})$ for the subset of pairs with this property. The group $G_{\mathbf{d}}$ acts on all three of the sets $\Rep_{\mathrm{fr}}(Q,\mathbf{d},\mathbf{n})$, $\Rep_{\mathrm{fr}}(J,\mathbf{d},\mathbf{n})$, and $\Rep_{\mathrm{fr}}^{\mathrm{cc}}(J,\mathbf{d},\mathbf{n})$ by 
\[
g\cdot(M,(f_j)_{j\in Q_0})=(g\cdot M,(f_jg_j^{-1})_{j\in Q_0})
\]
where $g=(g_j)_{j\in Q_0}\in G_{\mathbf{d}}$, and the orbits correspond to isomorphism classes of representations of the framed quiver $\widetilde{Q}$. We consider the moduli space 
\[
\mathcal{M}_{\mathbf{d}}^{\mathrm{cc}}(\widetilde{Q})=\Rep_{\mathrm{fr}}^{\mathrm{cc}}(J,\mathbf{d},\mathbf{n})/G_{\mathbf{d}}
\]
parametrizing cocyclic representations of $\widetilde{Q}$.

\subsection{Relation to quiver Grassmannians}

Finally, we relate the previous constructions involving framed quivers to the quiver Grassmannians used to categorify the canonical basis. We begin by recalling some basic ideas from the representation theory of quivers.

\begin{definition}
For each vertex $j$ of a quiver $Q$, let $E_j$ denote the simple representation at~$j$ defined by $\underline{\dim}\,E_j=\mathbf{e}_j$. Then the \emph{projective cover} of $E_j$ is the representation $P_j$ of~$Q$ such that $P_j(i)$ is the $\mathbb{C}$-vector space with basis given by the set of paths from $j$ to~$i$. The \emph{injective hull} of~$E_j$ is the representation $I_j$ of~$Q$ such that $I_j(i)$ is the dual of the $\mathbb{C}$-vector space with basis given by the set of paths from~$i$ to~$j$.
\end{definition}

Similarly, we define a module $I_j$ over the Jacobian algebra $J=\mathbb{C}Q/\mathfrak{a}$ by letting $I_j(i)$ be the dual of the $\mathbb{C}$-vector space with basis given by the images in~$J$ of paths from~$i$ to~$j$. Denoting by $I\otimes V$ the module $I\otimes V=\bigoplus_{j\in Q_0} I_j\otimes_\mathbb{C} V_j$, we have isomorphisms 
\[
(I\otimes V)(i)\cong\bigoplus_{j\in Q_0}I_j(i)\otimes_\mathbb{C} V_j\cong\bigoplus_{j\in Q_0}\bigoplus_{i\leadsto j}V_j
\]
where $i\leadsto j$ indicates the element of the Jacobian algebra represented by a path from~$i$ to~$j$.

The next result follows from a theorem of Fedotov~\cite{Fedotov}, which generalizes earlier work of Reineke~\cite{Reineke}.

\begin{theorem}[\cite{Fedotov}, Theorem~3.5]
\label{thm:Fedotov}
There is an isomorphism 
\[
\mathcal{M}_{\mathbf{d}}^{\mathrm{cc}}(\widetilde{Q})\cong\Gr_{\mathbf{d}}(I\otimes V)
\]
of the moduli space of cocyclic modules with the Grassmannian of subrepresentations of the representation $I\otimes V$ of~$Q$ satisfying the relations coming from the potential.
\end{theorem}

\begin{lemma}
\label{lem:modulelamination}
If $\ell$ is an integral lamination satisfying Assumption~\ref{assumption}, then there is an isomorphism $I\otimes V\cong M_{T,\ell}$ of modules over the Jacobian algebra.
\end{lemma}

\begin{proof}
We can represent the lamination $\ell$ by a collection of curves $c_0,\dots,c_N$ of weight~$\pm1$ that intersect the edges of~$T$ transversely in the minimal number of points. If $c_i$ connects points on the boundary of the surface, let $c_i'$ be the curve on the enlarged surface~$\overline{\Sigma}$ considered in Lemma~\ref{lem:productgvectors}. Otherwise, if $c_i$ is a closed curve, let us define $c_i'$ to be the curve~$c_i$, considered as a curve on~$\overline{\Sigma}$. One can check that the integer $n_j$ is exactly the negative shear parameter of the lamination formed by the curves $c_0',\dots,c_N'$ associated to the arc~$j$. If $\ell$ contains a closed loop, then the shear parameters take both positive and negative values, contradicting Assumption~\ref{assumption}. Therefore $\ell$ cannot contain a closed loop.

Suppose $j$ is an arc of~$T$ such that $n_j>0$. Let $q$ be the quadrilateral formed by the two triangles that share the edge~$j$. Let $c_{k_1}',\dots,c_{k_s}'$ be the curves that go across~$j$, connecting opposite sides of~$q$ as illustrated in Figure~\ref{fig:modifiedcurve}.
\begin{figure}[ht]
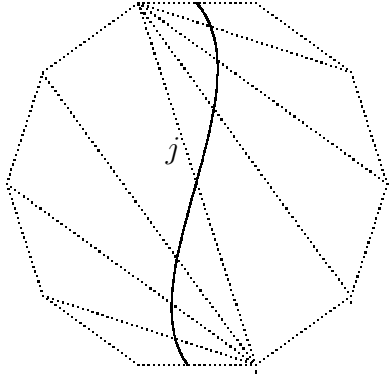
 \begin{center}
\[
\xy /l3pc/:
{\xypolygon10"A"{~:{(2,0):}~>{.}}};
{"A9"\PATH~={**@{.}}'"A4"},
{"A9"\PATH~={**@{.}}'"A5"},
{"A9"\PATH~={**@{.}}'"A6"},
{"A9"\PATH~={**@{.}}'"A7"},
{"A10"\PATH~={**@{.}}'"A4"},
{"A1"\PATH~={**@{.}}'"A4"},
{"A2"\PATH~={**@{.}}'"A4"},
(1,-1.9)*{}="1";
(1.1,1.9)*{}="2";
"1";"2" **\crv{(0.25,-1) & (1.75,1)};
(1.25,-0.35)*{j};
\endxy
\]
\caption{Geometry appearing in the proof of Lemma~\ref{lem:modulelamination}.\label{fig:modifiedcurve}}
\end{center} \end{figure}
If $c_{k_r}'$ connects opposite sides of another such quadrilateral, then it follows from the description in terms of shear parameters that we must have $n_i<0$ for some~$i$. Thus $c_{k_r}'$ cannot connect opposite sides of any quadrilateral other than~$q$ (see Figure~\ref{fig:modifiedcurve}).

Consider the arcs $i$ in the ideal triangulation~$T$ such that there exists a path from~$i$ to~$j$ in the quiver $Q=Q(T)$. These include all arcs that $\gamma_{k_r}$ intersects. If $i$ is any vertex such that $\gamma_{k_r}$ does not intersect $i$ but there exists a path $p$ from~$i$ to~$j$, then the path must contain two edges of one of the small counterclockwise oriented cycles in~$Q$. Call these edges $b$ and~$c$, and let $a$ be the remaining edge in the 3-cycle so that, up to cyclic equivalence, the potential can be written 
\[
W=abc+\dots.
\]
Then 
\[
\partial_aW=bc
\]
and so the image of $p$ in the quotient of $\mathbb{C}Q$ by the ideal generated by $\{\partial_aW:a\in Q_1\}$ is the zero class. Thus we see that $M_{T,\gamma_{k_r}}(i)\cong\bigoplus_{i\leadsto j}\mathbb{C}$. Taking a direct sum over~$r$ and~$j$, we find 
\[
M_{T,\ell}(i)\cong\bigoplus_j\bigoplus_{r=1}^s M_{T,\gamma_{k_r}}(i)\cong\bigoplus_j\bigoplus_{i\leadsto j}V_j\cong (I\otimes V)(i),
\]
and this in fact gives an isomorphism of modules.
\end{proof}

\begin{theorem}
Let $\sigma$ be a stability condition on the abelian category of modules over $J(Q(T),W(T))$. Let $\ell\in\mathcal{A}(\mathbb{Z}^t)$ be an integral lamination satisfying Assumption~\ref{assumption} with respect to~$T$. Then there is an isomorphism of varieties $\mathcal{M}_{\mathbf{d}}^{\mathrm{st}}(\widetilde{Q})\cong\Gr_{\mathbf{d}}(M_{T,\ell})$ for any dimension vector~$\mathbf{d}$. In particular, 
\[
\mathbb{I}(\ell)=\sum_{\mathbf{d}\in\mathbb{Z}^n, i\in\mathbb{Z}}(-1)^i\dim_{\mathbb{C}}\mathcal{H}^{\sigma,\ell}_{\mathbf{d},i}\cdot X_{\mathbf{d}}
\]
where $\mathcal{H}^{\sigma,\ell}_{\mathbf{d},i}=H^i(\mathcal{M}_{\mathbf{d}-\mathbf{h}}^{\mathrm{st}}(\widetilde{Q}),\mathbb{C})$.
\end{theorem}

\begin{proof}
By Proposition~\ref{prop:stablecyclic}, we know that a stable representation of the framed quiver $\widetilde{Q}$ is the same as a cocyclic representation, and therefore we have an isomorphism 
\[
\mathcal{M}_{\mathbf{d}}^{\mathrm{st}}(\widetilde{Q})\cong\mathcal{M}_{\mathbf{d}}^{\mathrm{cc}}(\widetilde{Q}).
\]
Applying the isomorphisms of Theorem~\ref{thm:Fedotov} and Lemma~\ref{lem:modulelamination}, we have 
\[
\mathcal{M}_{\mathbf{d}}^{\mathrm{cc}}(\widetilde{Q})\cong\Gr_{\mathbf{d}}(M_{T,\ell}).
\]
The theorem follows.
\end{proof}

\section*{Acknowledgments}
\addcontentsline{toc}{section}{Acknowledgements}

In writing this paper, I have benefitted from conversations with many people, including Tom~Bridgeland, Michele~Cirafici, Michele~Del~Zotto, Joseph~Karmazyn, Daniel~Labardini-Fragoso, Sven~Meinhardt, Andrew~Neitzke, Harold~Williams, and Yu~Zhou.

\end{document}